\documentclass[1 [leqno,11pt]{amsart}
\usepackage{amssymb, amsmath}
\allowdisplaybreaks[4]
\usepackage{color}
\usepackage{hyperref}

\let\pa=\partial
\let\f=\frac

\let\p=\partial
\let\om=\omega

\def\ka{\kappa}

\def\oo{\infty}

\def\cF{{\mathcal F}}

\def\cM{{\mathcal M}}

\def\eqdef{\buildrel\hbox{\footnotesize def}\over =}
\def\Z{\mathop{\mathbb Z\kern 0pt}\nolimits}
\def\N{\mathop{\mathbb N\kern 0pt}\nolimits}
\def\Q{\mathop{\mathbb Q\kern 0pt}\nolimits}
\def\R{{\mathop{\mathbb R\kern 0pt}\nolimits}}
\def\T{{\mathop{\mathbb T\kern 0pt}\nolimits}}

\def\dive{{\mathop{\rm div}\nolimits}\,}

\def\ini{{\rm in}}
\def\sgn{{\rm sgn}}

\def\iDx{\langle D_x\rangle}
\def\iD{\langle D_x, D_y\rangle}
\def\iDt{\langle D_x,D_y+tD_x\rangle}
\def\ik{\langle k\rangle}
\def\il{\langle l\rangle}
\def\ikl{\langle k-l\rangle}
\def\ikt{\langle k,\xi+kt\rangle}
\def\ilt{\langle l,\eta+lt\rangle}
\def\iklt{\langle k-l,\xi-\eta+(k-l)t\rangle}
\def\ifD{\langle \f1{D_x}\rangle}
\def\ifl{\langle \f1{l} \rangle}
\def\ifk{\langle \f1k\rangle}
\def\ifkl{\langle \f1{k-l}\rangle}
\def\hw{\hat{\omega}}
\def\Lw{\omega^{\rm L}}
\def\hLw{\hat{\omega}^{\rm L}}
\def\NLw{\omega^{\rm NL}}
\def\hNLw{\hat{\omega}^{\rm NL}}
\def\Lu{u^{\rm L}}
\def\NLu{u^{\rm NL}}
\def\NLp{\phi^{\rm NL}}
\def\hf{\hat{f}}

\def\e{\varepsilon}
\def\eqdefa{\buildrel\hbox{\footnotesize def}\over =}

\newcommand{\Rmnum}[1]{\uppercase\expandafter{\romannumeral #1} }

\newcommand{\beq}{\begin{equation}}
\newcommand{\eeq}{\end{equation}}
\newcommand{\ben}{\begin{eqnarray}}
\newcommand{\een}{\end{eqnarray}}
\newcommand{\beno}{\begin{eqnarray*}}
\newcommand{\eeno}{\end{eqnarray*}}

 \numberwithin{equation}{section}
\newcommand{\andf}{\quad\hbox{and}\quad}

\newtheorem{thm}{Theorem}[section]
\newtheorem{lem}{Lemma}[section]
\newtheorem{rmk}{Remark}[section]
\newtheorem{col}{Corollary}[section]
\newtheorem{prop}{Proposition}[section]

\begin{document}

\title[Two-dimensional Couette flow in the whole plane]{Stability threshold of the two-dimensional Couette flow in the whole plane}

\author{Hui Li}
\address[H. Li]{School of Mathematics, Sichuan University, Chengdu 610064, P. R. China}
\email{lihui92@scu.edu.cn}

\author{Ning Liu}
\address[N. Liu]{Academy of Mathematics $\&$ Systems Science, The Chinese Academy of Sciences, Beijing 100190, P. R. China. } 
\email{liuning16@mails.ucas.ac.cn}

\author{Weiren Zhao}
\address[W. Zhao]{Department of Mathematics, New York University Abu Dhabi, Saadiyat Island, P.O. Box 129188, Abu Dhabi, United Arab Emirates.}
\email{zjzjzwr@126.com, wz19@nyu.edu}

\date{\today}

\begin{abstract}
In this paper, we study the stability threshold for the two-dimensional Couette flow in the whole plane. Our main result establishes that the asymptotic stability threshold is at most $\frac{1}{3}+$ for Sobolev perturbations with additional control over low horizontal frequencies, aligning with the threshold results in periodic domains. As a secondary outcome of our approach, we also prove the asymptotic stability for perturbations in weak Sobolev regularity with size $\nu^{\frac{1}{2}}$. 
\end{abstract}

\maketitle

\setcounter{equation}{0}

\section{Introduction}

In this paper, we study the nonlinear stability of the Coutte flow $(y,0)$ for the two-dimensional Navier-Stokes equations in $\R^2$:
\begin{equation*}(\text{NS})\qquad
\left\{\begin{array}{l}
\displaystyle \pa_t v-\nu\Delta v +v\cdot\nabla v+\nabla p=0, \qquad (t,x,y)\in\R^+\times\R^2, \\
\displaystyle \dive v = 0, \\
\displaystyle  v|_{t=0}=v_{\ini}(x,y),
\end{array}\right.
\end{equation*}
where $v(t,x,y)=(v^1,v^2)$ denotes the velocity field, $p$ denotes the pressure, and $\nu$ is a small viscosity coefficient.

Let $u=v-(y,0)$ be the perturbation of the velocity, which satisfies
\begin{equation}
\quad \left\{\begin{array}{l}
\displaystyle \pa_t u-\nu\Delta u +y\p_x u +(u^2,0) +u\cdot\nabla u+\nabla p=0, \qquad (t,x,y)\in\R^+\times\R^2, \\
\displaystyle \dive u = 0, \\
\displaystyle  u|_{t=0}=u_{\ini}(x,y).
\end{array}\right.
\end{equation}
Let $\omega =\p_y u^1-\p_x u^2$ be the vorticity, which satisfies
\begin{equation}\label{eqs:w}
\quad \left\{\begin{array}{l}
\displaystyle \pa_t \om-\nu\Delta \omega +y\p_x \omega +u\cdot\nabla \omega=0, \qquad (t,x,y)\in\R^+\times\R^2, \\
\displaystyle  u = \nabla^\bot (-\Delta)^{-1} \omega =(-\p_y, \p_x)(-\Delta)^{-1} \omega, \\
\displaystyle  \omega|_{t=0}=\omega_{\ini}(x,y).
\end{array}\right.
\end{equation}

\subsection{Backgrounds}
The study of the stability of Couette flow has been a prominent topic in fluid mechanics since the pioneering works of Kelvin \cite{Kelvin1887}, Reyleigh \cite{Rayleigh1880}, Orr \cite{Orr1907} and Sommerfeld \cite{Sommerfeld1908}. In \cite{Kelvin1887}, Kelvin derived the solution formula to the following linearized equation of \eqref{eqs:w}: 
\begin{equation}
\p_t \omega +y\p_x \omega -\nu \Delta \omega=0 \andf \omega|_{t=0}=\omega_{\ini}.
\end{equation}
Let $\hw(t,k,\xi)$ denote the Fourier transform of $\omega(t,x,y)$. The Kelvin's solution is 
\begin{equation}
\hw(t,k,\xi)=\hw_{\ini}(k,\xi+kt) e^{-\nu {\int_0^t} |k|^2+|\xi+k(t-s)|^2 ds},
\end{equation}
which satisfies the following two linear estimates:
\begin{align}
&|\hw(t,k,\xi)| \leq C |\hw_\ini(k,\xi+kt)| e^{-c\nu^\f13 |k|^\f23 t}, \label{eq1.5a}
\\
&|\hat{\phi}(t,k,\xi)| \leq C \langle t\rangle^{-2}\f{1+|k|^2+|\xi+kt|^2}{|k|^4} |\hw_\ini(k,\xi+kt)|e^{-c\nu^\f13 |k|^\f23 t},\label{eq1.6}
\end{align}
where $\phi=\Delta^{-1}\omega$ is the associated stream function. The estimate \eqref{eq1.5a} is the enhanced dissipation, and the estimate \eqref{eq1.6} is the inviscid damping. Both stability mechanisms allow us to study the asymptotic stability of the full nonlinear system. 

On periodic domain $\T_x\times \R_y$, the Fourier frequency $k$ belongs to $\Z$. In such case, \eqref{eq1.5a} demonstrates that the nonzero modes of $\omega$ decay on a time scale $t\gtrsim \nu^{-\f13}$, significantly faster than the standard heat dissipation scale $t\gtrsim \nu^{-1}$. This phenomenon is known as the enhanced dissipation \cite{albritton2022enhanced}.
 Furthermore, \eqref{eq1.6} implies a polynomial decay of $\phi$, known as the inviscid damping and first observed by Orr in \cite{Orr1907}. For the inviscid case ($\nu=0$), the nonlinear inviscid damping was first proven by Bedrossian and Masmoudi \cite{BM2015}, see also \cite{IJ2018}. Recent progress in \cite{IJ2020, MasmoudiZhao2020} has shown that nonlinear inviscid damping holds for stable monotonic shear flows. 
 
These stability mechanisms are weaker on the whole plane $\R^2$. It is still an open problem whether the asymptotic stability of Couette flow holds for the Euler equation in the whole plane setting. There are two main difficulties: 
\begin{enumerate}
\item Both linear decay estimates are weaker since $k\in \mathbb{R}$: The linear enhanced dissipation estimate in \eqref{eq1.5a} indicates the enhanced dissipation rate slows for very low horizontal frequencies. Meanwhile, the linear inviscid damping estimate in \eqref{eq1.6} reveals that the additional polynomial decay in inviscid damping is accompanied by a singularity of $\f1{|k|}$.
\item  The nonlinear echos are stronger, which will be discussed in section \ref{sec:idea}. It motivates us to study the echos in the $\mathbb R^2$.
\end{enumerate}

This paper focuses on the stability problem for the Navier-Stokes equation with small viscosity $\nu>0$. Mathematically, Bedrossian, Germain and Masmoudi \cite{BGM2017} formulated the transition threshold problem:

{\it 
Given a norm $\|\cdot\|_X$, find a $\beta=\beta(X)$ so that 
\begin{itemize}
\item[] $\|u_{\ini}\|_{X}\leq \nu^\beta$ $\Rightarrow$ stability, enhanced dissipation and inviscid damping,
\item[] $\|u_{\ini}\|_X \gg \nu^\beta$ $\Rightarrow$ instability.
\end{itemize}
We then say $\beta$ is the transition threshold.
}

On the domain $\T_x\times\R_y$, the following important results are known:
\begin{itemize}
\item If $X$ is Gevrey class $2_-$, then \cite{BMV2016} shows $\beta\leq 0$, and \cite{DM2023} showed $\beta\geq 0$.
\item If $X$ is Sobolev space $H^{\log}_xL^2_y$, then \cite{BVW2018,MasmoudiZhao2020cpde} showed $\beta\leq \f12$ and \cite{LiMasmoudiZhao2022critical} showed $\beta\geq \f12$.
\item If $X$ is Sobolev space $H^\sigma$ ($\sigma\geq 2$), then \cite{MasmoudiZhao2019,wei2023nonlinear} showed $\beta\leq \f13$.
\item If $X$ is Gevrey class $\f1s$ with $s\in[0,\f12]$, then \cite{LMZ2022G} showed $\beta\leq \f{1-2s}{3(1-s)}$.
\end{itemize}

We also refer to the recent works \cite{BHIW2023,bedrossian2024uniform,CLWZ2020} for the stability threshold of the 2D Couette flow in a finite channel, and \cite{BGM2017,BGM2020,BGM2015,ChenWeiZhang2020,WeiZhang2020}
for the stability of the 3D Couette flow.

The goal of this paper is to study the optimal stability threshold in Sobolev-type spaces on the whole plane $\R^2$. More precisely, motivated by \eqref{eq1.5a} and \eqref{eq1.6}, we ask the following question:

{\it 
Given two norms $\|\cdot\|_X$ and $\|\cdot\|_Y$ $(X, Y\subset L^2(\R^2))$, find a $\beta=\beta(X, Y)$ as small as possible such that for the initial data satisfying $\|\omega_{\ini}\|_X\ll \nu^\beta$, and for all $T>0$,
\begin{equation}\label{question}
\sup_{t\in[0,T]}\left\|e^{c\nu^\f13|D_x|^\f23t}\omega(t)\right\|_{Y} + \left\|e^{c\nu^\f13|D_x|^\f23t}\p_x u(t)\right\|_{L^2([0,T];Y)}
\leq C \|\omega_{\ini}\|_{X}.
\end{equation}
}

Recent works \cite{wang2024transition} and \cite{arbon2024} provided some insights into this question. The best stability threshold achieved so far is the smallness $\nu^\f12(1+\log\f1\nu)^{-\f12} $ established in \cite{arbon2024}.
Notably, their results do not directly address question  \eqref{question}. In \cite{arbon2024}, the authors continued the methods in \cite{ZelatiGallay2023} to prove decay for $\hw(t,k,\xi)$ at enhanced dissipation rate $e^{-c\nu^\f13|k|^\f23 t}$ when $|k|\geq \nu$ and a decay rate $e^{-c\nu^{-1}|k|^2t}$ for lower frequencies $|k|\leq \nu$, which they call the Taylor dispersion and refer to \cite{aris1956dispersion, taylor1954dispersion}. Since enhanced dissipation is strictly faster than Taylor dispersion for the frequencies $|k|\leq \nu$, the estimates \eqref{question} are hard to prove.

\subsection{Main results}
In this paper, we always denote $\langle a\rangle=(1+a^2)^\f12$ and $\langle a, b\rangle=(1+a^2+b^2)^\f12$. For any function $\varphi: \mathbb{R}\to \mathbb{R}$, we denote the associated Fourier multiplier by $\varphi(D_x)$, namely,
\begin{align*}
\widehat{(\varphi(D_x)f)}(k)=\varphi(k)\hat{f}(k). 
\end{align*}

Our main result is to prove the stability threshold $\f13+$:
\begin{thm}\label{thm2}
For any $\delta>0$, $0<\epsilon<\f12$, there exist $0<c, \nu_0<1$ such that for all $0<\nu\leq \nu_0$, if the initial data $\omega_{\ini}$ satisfies
\begin{equation}\label{smallness2}
\left\|\iD^{6} \ifD^4 \omega_{\ini} \right\|_{L^2}+\left\|\iD^5 \ifD^4 \omega_{\ini} \right\|_{L^1}\leq  \nu^{\f13+\delta},
\end{equation}
then \eqref{eqs:w} has a unique global solution $\omega(t)$ satisfying the following  stability estimates:
\begin{itemize}
\item[(1)] enhanced dissipation estimate:
\begin{equation}\label{eq1.7}
\sup_{t\in[0,\oo)}\left\| e^{c\nu^\f13 \lambda(D_x) t} \iDx \ifD^\epsilon \omega(t)\right\|_{L^2} \leq C \nu^{\f13+\delta},
\end{equation}
\item[(2)] inviscid damping estimate:
\begin{equation}\label{eq1.8}
\left(\int_0^\oo\left\| e^{c\nu^\f13 \lambda(D_x) t} \iDx \ifD^\epsilon \p_x u(t)\right\|_{L^2}^2\ dt\right)^\f12 \leq C \nu^{\f13+\delta},
\end{equation}
\end{itemize}
where we denote \begin{equation}\label{def:lambda}
\lambda(k)=\min(1,|k|^\f23).
\end{equation}
\end{thm}
As a side result, we improve Theorem 1.1 in \cite{arbon2024}:
\begin{thm} \label{thm1}
For any $m>\f12, 0<\epsilon<\f12$, there exists $0<c,\e_0<1$ such that for all $0<\nu\leq 1$ and $0<\e\leq \e_0$, if the initial data $\omega_{\ini}$ satisfies
\begin{equation}\label{smallness1}
\left\|\iDx^m \ifD^\epsilon \omega_{\ini} \right\|_{L^2}\leq \e \nu^\f12,
\end{equation}
then \eqref{eqs:w} has a unique global solution $\omega(t)$ satisfying the following  stability estimates:
\begin{itemize}
\item[(1)] enhanced dissipation estimate:
\begin{equation}\label{eq1.4}
\sup_{t\in[0,\oo)}\left\| e^{c\nu^\f13 |D_x|^\f23 t} \iDx^m \ifD^\epsilon \omega(t)\right\|_{L^2} \leq C \e \nu^\f12,
\end{equation}
\item[(2)] inviscid damping estimate:
\begin{equation}\label{eq1.5}
\left(\int_0^\oo\left\| e^{c\nu^\f13 |D_x|^\f23 t} \iDx^m \ifD^\epsilon \p_x u(t)\right\|_{L^2}^2\ dt\right)^\f12 \leq C \e \nu^\f12.
\end{equation}
\end{itemize}
\end{thm}

Some remarks are given in order. 
\begin{rmk}
 The additional $\nu^\delta$ in \eqref{smallness2} can be replaced by logarithmic corrections.
\end{rmk}

\begin{rmk}
{  The distinction between the decay rate $\nu^\f13\lambda(k)$ and 
the one in \eqref{question} occurs only in the range where $|k|>1$.  When $|k|>1$, since both the rates are already faster than the enhanced dissipation rate in the torus domains, one can modify the question \eqref{question} slightly.
}\end{rmk}

\begin{rmk}\label{rmk1.3}
{  The assumption \eqref{smallness2} is far from sharp conditions, but we take it for simplicity of proof. The order $\ifD^4$ we chose in \eqref{smallness2} can help us to prove a better inviscid damping result such that $\omega=\Lw+\NLw$ with $\Lw$ satisfying the point-wise decay as in \eqref{eq1.6}:
$$
\left\|e^{c\nu^\f13|D_x|^\f23 t} \iDt^4 \ifD^\epsilon  \Delta^{-1}\Lw\right\|_{L^2}\leq C \langle t\rangle^{-2} \nu^{\f13+\delta},
$$
and $\NLu=\nabla^\perp\Delta^{-1}\NLw$ satisfying the $L^2_t$ decay as \eqref{eq1.8} but with a size much smaller than initial perturbation:
$$
\left(\int_0^\oo\left\| e^{c\nu^\f13 \lambda(D_x) t} \iDx \ifD^\epsilon \p_x \NLu(t)\right\|_{L^2}^2\ dt\right)^\f12 \leq C \nu^{\f23+2\delta}.
$$
Thus, we have a good control of the terms from the nonlinear interactions between two linear solutions, see \eqref{source term:LL} and \eqref{eq4.19}.
With some small modifications of our proof of Theorem \ref{thm2}, one can easily weaken the assumption \eqref{smallness2} to 
$$
\left\|\iD^{6} \ifD^2 \omega_{\ini} \right\|_{L^2\cap L^1}\leq  \nu^{\f13+\delta}.
$$

A toy model is introduced and discussed in the next subsection to show that some additional control of low frequency, such as at least $\|\ifD^{\f12} \omega_{\ini}\|_{L^2}$ or $\|\omega_{\ini}\|_{L^1}$, is necessary for proving the $\f13$ stability threshold. Since the assumption of Theorem \ref{thm2} is not optimal, we would like to propose the following open problem:

{\it 
If $X= H^m(\R^2)\cap W^{m,1}(\R^2)\cap \dot{H}^{-1}(\R^2)$ for some large enough $m$ and $Y=L^2(\R^2)$, does $\beta \geq \f13$ imply stability?
}
}\end{rmk}

\begin{rmk}
Comparing Theorem \ref{thm1} presented above with Theorem 1.1 in \cite{arbon2024}, our result makes improvements in the following aspects:
\begin{itemize}
\item In Theorem 1.1 of \cite{arbon2024}, the time decay rate is given as the Taylor dispersion rate $\f{|k|^2}\nu$ for small horizontal frequencies $|k|\leq \nu$. We improve it to the enhanced dissipation rate $\nu^\f13|k|^\f23$ for all frequencies.
\item We remove the logarithmic correction in the smallness assumption $\nu^\f12(1+\log \f1\nu)^{-\f12}$ in Theorem 1.1 of \cite{arbon2024}.
\item In Theorem 1.1 of \cite{arbon2024}, the authors assumed a lower order norm $\|\omega_{\ini,k}\|_{L^\oo_kL^2}$ to prevent potential concentration near frequency $k=0$. Here, we weaken the assumption to be some norm $\|\iDx^m \ifD^\epsilon \omega_{\ini} \|_{L^2}$, which is nearly only $H^{m,0}$, as $\epsilon\to 0_+$. 
\end{itemize}
We emphasize that the method in \cite{arbon2024} can deal with different domains with Navier-slip boundary conditions. The Fourier multiplier method in this paper may not be easy to directly apply in those cases. It is interesting to study the stability threshold problem \eqref{question} in the finite channel setting $\mathbb{R}\times [0,1]$ with different boundary conditions. 
\end{rmk}

\subsection{Challenges and ideas}\label{sec:idea}

In this subsection, we introduce a simplified toy model to highlight the additional challenges encountered when addressing the stability problem on $\R^2$, as opposed to prior studies on periodic domains.

Motivated by some observations in \cite{BM2015}, we reformulate \eqref{eqs:w} in new variables $(z,y)=(x-ty,y)$ as 
$$
\p_t f-\nu\Delta_L f= \nabla^\perp(-\Delta_L)^{-1} f \cdot \nabla f,
$$
where $f(t,z,y)=\omega(t,z+ty,y)$ and $\Delta_L=\p_z^2+(\p_y-t\p_z)^2$. Applying the Fourier transform yields the equation
\begin{equation}\label{eq1.10}
\begin{aligned}
&\p_t \hf(t,k,\xi)-\nu(|k|^2+|\xi-kt|^2)\hf(t,k,\xi)\\
&\qquad\qquad=\int_{\R^2} \f{\eta(k-l)-l(\xi-\eta)}{|l|^2+|\eta-lt|^2} \hf(t,l,\eta)\hf(t,k-l,\xi-\eta)dld\eta.
\end{aligned}
\end{equation}

Following the approach of prior works (e.g., (1.11) in \cite{LMZ2022G}), we consider the following toy model on the Fourier side:
\begin{equation}\label{eq1.11}
\p_t \hf(t,k,\xi)-\nu(|k|^2+|\xi-kt|^2)\hf(t,k,\xi)=\int_{\R} \f{\xi(k-l)}{|l|^2+|\xi-lt|^2} \hf(t,l,\xi)\hf(t,k-l,0)dl.
\end{equation}

An easy calculation shows that for any fixed frequecies $(k, l,\xi)$, the coefficient $\f{\xi(k-l)}{|l|^2+|\xi-lt|^2}$ attains its maximum value at $t=\frac{\xi}{l}$. This time is called the Orr critical time. In the periodic setting $x\in \mathbb{T}$, a cascade growth happens near each critical time \cite{BM2015}, namely, the critical time interval $[\frac{2\xi}{2l+1}, \frac{2\xi}{2l-1}]$ which contains only one critical time. Since $l\in\mathbb{Z}\setminus\{0\}$, for fixed $\xi>0$, the critical time intervals are contained in $(0,2\xi]$ which is a finite time interval. In our case, $x\in \mathbb{R}$, $l\in \mathbb{R}$, there are two natural difficulties: 1. The critical time is continuous in $l$, which leads to a strong overlap between two critical-time intervals. To overcome this difficulty and to capture the growth from nonlinear interactions, we design the time-dependent Fourier multiplier, see $\mathcal{M}_3$ in section 2, by fuzzifying the critical time intervals. 2. The previously mentioned cascade does not stop for $t\geq 2\xi$. Here, we use the enhanced dissipation effect, which offers additional smallness for the large time regime and stops the cascades. It is worth mentioning that the asymptotic (in)stability of Couette flow in the whole space $\mathbb{R}^2$ remains open. We refer to \cite{BM2015, LMZ2022G, MasmoudiZhao2019} for more discussions about constructing time-dependent Fourier multipliers to capture the growth from nonlinear interactions. 

Now let us discuss the new difficulties from lower frequencies. 
This toy model captures the case $|\xi-\eta|\ll |\eta|\sim|\xi|$ in \eqref{eq1.10}. Previous works primarily focus on the region that $|k-l|\ll |l|$, where the most challenging case is $|k-l|=1$. However, for our problem on $\R^2$, a more difficult situation arises when $|l|\ll |k-l|$, due to the possible singularity near $|l|=0$.

Now, we do some detailed analysis for \eqref{eq1.11}. Let $\epsilon, \beta>0$. Assuming the initial data satisfies a smallness condition like 
$$
\|\ifk^\epsilon  \hf_{\ini}(k,\xi)\|_{L^2}\leq \nu^{\beta},
$$ 
and then we expect
$$
\| e^{c\nu^\f13|k|^\f23t}\ifk^\epsilon \hf(t,k,\xi)\|_{L^2} \lesssim \nu^\beta.
$$
In a special region where $|l|\in [\f1N,\f2N]$ for large $N$ and $|k|, |k-l|\in [1,2]$, the right hand side in \eqref{eq1.11} may exhibit its worst behaviour when
$$
\hf(t,l,\xi)\approx \nu^\beta e^{-c\nu^{\f13}N^{-\f23}t} N^{\f12-\epsilon}\chi_{[\f1{N},\f2N]}(l) \andf \hf(t,k-l,0)\approx \nu^{\beta}e^{-c\nu^\f13 t} \chi_{[1,2]}(k-l).
$$
The fraction in \eqref{eq1.11} can be simplified in this region as 
$$
\f{\xi(k-l)}{|l|^2+|\xi-lt|^2} \approx N^2|\xi|\f{1}{1+|t-\f{\xi}l|^2},
$$
where the time integrability implies that the main contribution comes from a time interval like $[\f\xi{l}-1,\f\xi{l}+1]$. Consequently, given any $\xi$, for $t$ large enough, the worst situation considered above for the toy model will result in something like
$$
\int_{\f\xi{l}-1}^{\f\xi{l}+1} \int_{\f1N}^{\f2N} \f{\xi(k-l)}{|l|^2+|\xi-lt|^2} \hf(t,l,\xi)\hf(t,k-l,0)dl dt 
\approx \nu^{2\beta} N^{\f32-\epsilon} |\xi| e^{-c\nu^\f13 N|\xi|},
$$
where we used $t\approx N|\xi|$. Recalling that the above integral results in a function defined on $k\in[1,2]$, the question of whether the function remains $\nu^\beta$ small in $L^2$  boils down to examining whether $\nu^{\beta} N^{\f32-\epsilon} |\xi| e^{-c\nu^\f13 N|\xi|}$ is bounded. 

When $\beta=\f12$, since $N$ and $|\xi|$ in this toy model is large, one has 
$$
\nu^{\f12} N^{\f32-\epsilon} |\xi| e^{-c\nu^\f13 N|\xi|} \lesssim N^{-\epsilon}|\xi|^{-\f12}\lesssim 1,
$$
which indicates that $\epsilon$ can be, in principle, taken to be zero. However, since the analysis above focuses only on the region $|l|\in [\f1N,\f2N]$, summing up these challenging contributions across all dyadic intervals in $N$ shows the necessity of a slightly positive $\epsilon$. This partially explains why Theorem \ref{thm1} is provable under the assumption \eqref{smallness1} for any $\epsilon\in(0,\f12)$.

In contrast, when $\beta=\f13$, the following estimate
$$
\nu^{\f13} N^{\f32-\epsilon} |\xi| e^{-c\nu^\f13 N|\xi|} \lesssim N^{\f12-\epsilon}
$$
implies that $\epsilon$ must be taken larger than $\f12$. Technically, the multiplier $\ifD^\f12$ is very difficult for the nonlinear estimates, so we can only maintain the special structure for $\epsilon<\f12$. This motivates the assumption of $\nu^{\f13+\delta}$ smallness in Theorem \ref{thm2}. Also, one can see from the proof near \eqref{eq4.22} that some difficulty requires us to take $\epsilon$ close enough to $\f12$ with respect to $\delta$.

The above discussion of the toy model does not cover all the difficulties in the problem. For instance, when considering the region $|l|\ll|k-l|$ in \eqref{eq1.11}, it seems that, except the $\f1{|l|}$ singularity, the additional challenge of controlling the derivative $|k-l|$ arises. 
This difficulty resembles a transport term in the $x$-direction and one will face an extra $\f1{|l|}$ singularity while applying the point-wise inviscid damping technique as in the torus case. Therefore, identifying the optimal function space for the stability threshold $\f13$ remains an open problem.

In our result Theorem \ref{thm2}, we used a very strong function space with Fourier weight $\ifD^4$ in the assumption \eqref{smallness2}. While this assumption is far from optimal, it significantly simplifies the proof. However, in the nonlinear estimates, we can only transfer at most Fourier weights of the order $\ifD^{\f12-}$ as in the proof of Theorem \ref{thm1}. To make full use of \eqref{smallness2}, we use the idea of quasi-linearization to decompose $\omega=\Lw+\NLw$ with $\Lw$ solve the linear equation \eqref{eqs:Lw} which can transfer all the assumptions on initial data, as explained in Remark \ref{rmk1.3}. This quasi-linearization plays a crucial role in our analysis, see \cite{CLWZ2020, niuzhao2024, Wei2023quasilinear, ZhaiZhao2022} for other applications. 

\subsection{Outline of the paper and Notations}
We now sketch the structure of this paper.

In section \ref{section 2}, we introduce some Fourier multipliers and derive some linear estimates.

In section \ref{section 3}, we present the proof of Theorem \ref{thm1}.

In section \ref{section 4}, we first introduce the quasi-linear decomposition $\omega=\Lw+\NLw$, and then do energy estimates for $\NLw$ in two time scales $[0,\nu^{-\f16}]$ and $[\nu^{-\f16},+\oo )$ separately, and finally prove Theorem \ref{thm2}. 

Let us end this section with some notations that will be used
throughout this paper.

\noindent{\bf Notations:}   For $a\lesssim b$,
we mean that there is a uniform constant $C,$ which may be different in each occurrence, such that $a\leq Cb$. We shall denote by~$(a|b)$
the $L^2(\R^2)$ inner product of $a$ and $b.$  
Given a function $f(x,y)$ on $\R_x\times\R_y$, we shall denote $f_k(y)=\cF_{x\rightarrow k}(f)(k,y)$ the $k-$th horizontal Fourier modes of $f$, and $\hat{f}_k(\xi)=\hat{f}(k,\xi)$ the Fourier transform of $f$ with respect to both $x$ and $y$ variables.
 Finally, we denote $L^r_T(L^p)$ the space $L^r([0,T];
L^p(\R^2),$ and denote $L^r_{[T_1,T_2]}(L^p)$ the space $L^r([T_1,T_2];
L^p(\R^2)$. 

\section{Linear estimates}\label{section 2}
In this section, we introduce the following multipliers:
\begin{align*}
 \cM_1(k,\xi) &\eqdefa \arctan(\nu^\f13 |k|^{-\f13}\sgn(k)\xi)+\f\pi2,\\
 \cM_2(k,\xi) &\eqdefa \arctan(\f\xi{k})+\f\pi2,\\
 \cM_3(t,k,\xi) &\eqdefa \int_{\R} \ifl^{-1-\kappa} \f{1}{|l|^2} (\sgn(l)\arctan(\f{\xi+t(k-l)}{1+|k-l|+|l|})+\f\pi2) dl,
\end{align*}
where $\kappa>0$ is some positive constant to be chosen.

On the Fourier side, these multipliers can be understood as some kind of `ghost weight', which are bounded weights providing additional dissipation properties. The basic ideas come from:
\begin{equation}
2\Re \big( (\p_t+y\p_x)f \big| \cM_i f\big)=\f{d}{dt}\|\sqrt{\cM_i}f\|_{L^2}^2 +\int_{\R^2} (-\p_t+k\p_\xi)\cM_i(t,k,\xi) |\hat{f}|^2 dkd\xi.
\end{equation}

The enhanced dissipation multiplier $\cM_1$ and the inviscid damping multiplier $\cM_2$ are constructed in a standard way as many previous works, see for example \cite{BGM2017,BVW2018,DengWuZhang2021}. The usage of these two Fourier multipliers is explained by the following Lemma:
\begin{lem}
{  For smooth enough function $f$ on $\R^2$, one has
\begin{align}
\label{eq:M1}
&\int_{\R^2} k\p_\xi\cM_1(t,k,\xi) |\hat{f}(k,\xi)|^2 dkd\xi \geq \f{\nu^\f13}4\| |D_x|^\f13 f\|_{L^2}^2-\f{\nu}2\|\p_y f\|_{L^2}^2, \\ \label{eq:M2}
&\int_{\R^2} k\p_\xi\cM_2(t,k,\xi) |\hat{f}(k,\xi)|^2 dkd\xi \geq \| \p_x \nabla \Delta^{-1} f\|_{L^2}^2.
\end{align}
}\end{lem}
\begin{proof}
The estimate \eqref{eq:M1} follows from 
$$
k\p_\xi \cM_1(k,\xi)=\f{\nu^\f13 |k|^\f23}{1+\nu^\f23 |k|^{-\f23}|\xi|^2}\geq \f14 \nu^\f13 |k|^\f23-\f12\nu |\xi|^2,
$$
where we used $\f1{1+A}+\f{A}2\geq \f14$ for $A=\nu^\f23 |k|^{-\f23}|\xi|^2\geq0$.

Similarly, direct computations show
$$
k\p_\xi \cM_2(k,\xi)= \f{k^2}{k^2+\xi^2},
$$
which implies \eqref{eq:M2}. This finishes the proof.
\end{proof}

The construction of the third multiplier $\cM_3$ is related to some ideas in \cite{wei2023nonlinear}. Let us denote 
\begin{equation}\label{eq:M3}
\Upsilon (t,k,\xi)=(-\p_t+k\p_\xi)\cM_3=\int_{\R} \ifl^{-1-\kappa} \f{1}{|l|} \f{1+|k-l|+|l|}{(1+|k-l|+|l|)^2+|\xi+t(k-l)|^2} dl.
\end{equation}
In the work \cite{wei2023nonlinear} on domain $\T_x\times\R_y$, the authors constructed a similar $\cM_3$ with $\kappa=0$ and the integration of $k\in\R$ replaced by the summation over $\Z$.
Here, in our case, we have to take some $\ka>0$, so that both $\cM_3$ and $\Upsilon$ are bounded operators. This weight $\cM_3$ gives us a new kind of inviscid dissipation with a different structure from inviscid damping from $\cM_2$. The multiplier $\cM_3$ is designed to control the growth of the reaction term caused by echo cascades, see around \eqref{eq4.22}.

\section{Proof of Theorem \ref{thm1}}\label{section 3}
In this section, we shall present the proof of Theorem \ref{thm1}. We shall denote $\cM(k,\xi)=\cM_1(k,\xi)+\cM_2(k,\xi)+1$ throughout this proof.

\begin{prop}\label{prop3.1}
{  For $m>\f12$, $0<\epsilon<\f12$ and $c<\f1{16(1+2\pi)}$, there exists $C$ depending only on $m$ and $\epsilon$ such that for any smooth solutions of \eqref{eqs:w} on $[0,T]$, it holds
\begin{equation} \label{eq:prop3.1}
\begin{aligned}
&\|e^{c\nu^\f13|D_x|^\f23 t}\iDx^{m}\ifD^{\epsilon} \omega(t)\|_{L^\oo_T L^2}^2
+\bigl(1-C\|e^{c\nu^\f13|D_x|^\f23 t}\iDx^{m}\ifD^{\epsilon} \omega(t)\|_{L^\oo_T L^2}\bigr) \\
&\qquad \times\bigl(\nu\|e^{c\nu^\f13|D_x|^\f23 t}\iDx^{m}\ifD^{\epsilon} \nabla \omega\|_{L^2_T L^2}^2
+\nu^\f13 \|e^{c\nu^\f13|D_x|^\f23 t}\iDx^{m}\ifD^{\epsilon} |D_x|^\f13\omega\|_{L^2_TL^2}^2
 \\
&\qquad\qquad+\|e^{c\nu^\f13|D_x|^\f23 t}\iDx^{m}\ifD^{\epsilon} \p_x \nabla \phi\|_{L^2_TL^2}^2\bigr)
\leq  C\|\iDx^{m}\ifD^{\epsilon} \omega_{\ini}\|_{L^2}^2.
\end{aligned}
\end{equation}
}\end{prop}
\begin{proof}
We take the inner product of \eqref{eqs:w} with $\cM e^{2c\nu^\f13|D_x|^\f23 t}\iDx^{2m}\ifD^{2\epsilon} \omega$, and infer that
\begin{align*}
&\f{d}{dt}\|\sqrt{\cM}e^{c\nu^\f13|D_x|^\f23 t}\iDx^{m}\ifD^{\epsilon} \omega\|_{L^2}^2
-2c\nu^\f13 \|\sqrt{\cM}e^{c\nu^\f13|D_x|^\f23 t}\iDx^{m}\ifD^{\epsilon} |D_x|^\f13\omega\|_{L^2}^2 \\
&+2\nu\|\sqrt{\cM}e^{c\nu^\f13|D_x|^\f23 t}\iDx^{m}\ifD^{\epsilon} \nabla \omega\|_{L^2}^2
+\int_{\R^2} k\p_\xi\cM(k,\xi) e^{2c\nu^\f13|k|^\f23 t}\langle k\rangle^{2m}\ifk^{2\epsilon} |\hw(k,\xi)|^2 dkd\xi \\
&\qquad=-2\Re \bigl( u\cdot\nabla \omega \big| \cM e^{2c\nu^\f13|D_x|^\f23 t}\iDx^{2m}\ifD^{2\epsilon} \omega \bigr).
\end{align*}
Thanks to \eqref{eq:M1} and \eqref{eq:M2}, we use the fact that $1\leq \cM(k,\xi)\leq 1+2\pi$ to deduce that for $c<\f1{16(1+2\pi)}$, we have 
\begin{equation}\label{eq3.1}
\begin{aligned}
&\f{d}{dt}\|\sqrt{\cM}e^{c\nu^\f13|D_x|^\f23 t}\iDx^{m}\ifD^{\epsilon} \omega\|_{L^2}^2
+\nu\|e^{c\nu^\f13|D_x|^\f23 t}\iDx^{m}\ifD^{\epsilon} \nabla \omega\|_{L^2}^2
 \\
&+\f{\nu^\f13}{16} \|e^{c\nu^\f13|D_x|^\f23 t}\iDx^{m}\ifD^{\epsilon} |D_x|^\f13\omega\|_{L^2}^2
+\|e^{c\nu^\f13|D_x|^\f23 t}\iDx^{m}\ifD^{\epsilon} \p_x \nabla \phi\|_{L^2}^2 \\
&\qquad \leq 2|\Re \bigl( u\cdot\nabla \omega \big| \cM e^{2c\nu^\f13|D_x|^\f23 t}\iDx^{2m}\ifD^{2\epsilon} \omega \bigr)|,
\end{aligned}
\end{equation}
where $\phi=\Delta^{-1}\omega$ denotes the stream function.

Now, we try to estimate the nonlinear terms 
$$\int_{\R^2} e^{2c\nu^\f13|k|^\f23 t}\ik^{2m}\ifk^{2\epsilon} \bigl(\int_\R \cM(k,D_y)\omega_k \cdot(\p_y \phi_l \cdot(k-l)\omega_{k-l}-l\phi_l \cdot \p_y\omega_{k-l})dy \bigr)dkdl,$$
where $\cM(k,D_y)$ can be understood as a bounded operator on $L^2_y$.

\begin{itemize}
\item For the term with $\p_y \phi_l \cdot(k-l)\omega_{k-l}$.
\end{itemize}

When $\f{|k-l|}2\leq|k|\leq 2|k-l|$, we deduce from $\ik^m\ifk^\epsilon\lesssim \ikl^m\ifkl^\epsilon$ and $|k-l|^\f13\lesssim |k|^\f13$ that
\begin{align*}
&\Big|\int_{\f{|k-l|}2\leq|k|\leq 2|k-l|} e^{2c\nu^\f13|k|^\f23 t}\ik^{2m}\ifk^{2\epsilon} \bigl(\int_\R \cM(k,D_y)\omega_k \cdot\p_y \phi_l \cdot(k-l)\omega_{k-l}dy \bigr)dkdl \Big|\\
 \lesssim &\big\|e^{c\nu^\f13|k|^\f23 t}\ik^{m}\ifk^{\epsilon}|k|^\f13 \|\omega_k\|_{L^2_y}\big\|_{L^2_{k}} \big\|e^{c\nu^\f13|k-l|^\f23 t}\ikl^{m}\ifkl^{\epsilon}|k-l|^\f23 \|\omega_{k-l}\|_{L^2_y}\big\|_{L^2_{k-l}}  \\
&\qquad \qquad \times \big\|e^{c\nu^\f13|l|^\f23 t} \|\p_y \phi_l \|_{L^\oo_y}\big\|_{L^1_{l}}  \\
\lesssim & \nu^{-\f12}(\nu^\f16\|e^{c\nu^\f13|D_x|^\f23 t}\iDx^{m}\ifD^{\epsilon} |D_x|^\f13\omega\|_{L^2})^{\f32}
(\nu^\f12\|e^{c\nu^\f13|D_x|^\f23 t}\iDx^{m}\ifD^{\epsilon} \p_x\omega\|_{L^2})^\f12  \\
&\qquad\qquad\times\|e^{c\nu^\f13|D_x|^\f23 t}\iDx^{m}\ifD^{\epsilon} \omega\|_{L^2},
\end{align*}
where we used from Gagliardo-Nirenberg-Sobolev inequality that
\begin{equation}\label{eq3.2}
\| f_l\|_{L^\oo_y}\lesssim \|f_l\|_{L^2_y}^\f12 \|\p_y f_l\|_{L^2_y}^\f12 \lesssim |l|^{-\f12}\|\nabla_l f_l\|_{L^2_y},
\end{equation}
and therefore
\begin{align*}
&\big\|e^{c\nu^\f13|l|^\f23 t} \|\p_y \phi_l \|_{L^\oo_y}\big\|_{L^1_{l}} \lesssim \big\|e^{c\nu^\f13|l|^\f23 t} |l|^{-\f12}\|\p_y \nabla_l\phi_l \|_{L^2_y}\big\|_{L^1_{l}} \\
&\lesssim \big\|e^{c\nu^\f13|l|^\f23 t}  \il^m \ifl^\epsilon\|\p_y \nabla_l\phi_l \|_{L^2_y}\big\|_{L^2_{l}}\||l|^{-\f12}\il^{-m}\ifl^{-\epsilon}\|_{L^2_l}
\lesssim \|e^{c\nu^\f13|D_x|^\f23 t}\iDx^{m}\ifD^{\epsilon} \omega\|_{L^2}.
\end{align*}

When $2|k-l|<|k|$, one can derive $\f{|k|}2\leq |l|\leq 2|k|$ and therefore $\ik^m\ifk^\epsilon\lesssim \il^m\ifl^\epsilon$, $|k-l|\lesssim |l|$. Then, we use Young's inequality to get
\begin{align*}
&\Big|\int_{2|k-l|<|k|} e^{2c\nu^\f13|k|^\f23 t}\ik^{2m}\ifk^{2\epsilon} \bigl(\int_\R \cM(k,D_y)\omega_k \cdot\p_y \phi_l \cdot(k-l)\omega_{k-l}dy \bigr)dkdl\Big| \\
 \lesssim &\big\|e^{c\nu^\f13|k|^\f23 t}\ik^{m}\ifk^{\epsilon} \|\omega_k\|_{L^2_y}\big\|_{L^2_{k}} \big\|e^{c\nu^\f13|l|^\f23 t}\il^m \ifl^\epsilon |l|\|\p_y \phi_l \|_{L^2_y}\big\|_{L^2_{l}}  \big\|e^{c\nu^\f13|k-l|^\f23 t} \|\omega_{k-l}\|_{L^\oo_y}\big\|_{L^1_{k-l}} \\
\lesssim & \nu^{-\f12}\|e^{c\nu^\f13|D_x|^\f23 t}\iDx^{m}\ifD^{\epsilon} \omega\|_{L^2} \|e^{c\nu^\f13|D_x|^\f23 t}\iDx^{m}\ifD^{\epsilon} \p_x\p_y\phi\|_{L^2}
  \\
&\qquad\qquad\times\nu^\f12\|e^{c\nu^\f13|D_x|^\f23 t}\iDx^{m}\ifD^{\epsilon} \nabla\omega\|_{L^2},
\end{align*}
where we used from \eqref{eq3.2} that
\begin{align*}
&\big\|e^{c\nu^\f13|k-l|^\f23 t} \|\omega_{k-l}\|_{L^\oo_y}\big\|_{L^1_{k-l}} \lesssim \big\|e^{c\nu^\f13|k-l|^\f23 t} |k-l|^{-\f12}\| \nabla_{k-l}\omega_{k-l}\|_{L^2_y}\big\|_{L^1_{k-l}} \\
&\lesssim \big\|e^{c\nu^\f13|k-l|^\f23 t}  \ikl^m \ifkl^\epsilon\| \nabla_{k-l}\omega_{k-l} \|_{L^2_y}\big\|_{L^2_{k-l}}\||k-l|^{-\f12}\ikl^{-m}\ifkl^{-\epsilon}\|_{L^2_{k-l}} \\
&\lesssim \|e^{c\nu^\f13|D_x|^\f23 t}\iDx^{m}\ifD^{\epsilon} \nabla\omega\|_{L^2}.
\end{align*}

The most difficult case is that $2|k|<|k-l|$, and in such case $\f{|k-l|}2\leq |l|\leq 2|k-l|$, we have to use $L^2_k$ integrability to overcome the singularity of $\ifk^\epsilon$. This is the reason we can only consider $\epsilon<\f12$. For more details, by discussing the value of $\epsilon\in (0,\f12)$, we can find 
\begin{equation}\label{eq3.2a}
|k-l|^{-\f12}\leq \ifkl^\f12\leq \ifl^\epsilon\ifkl^\epsilon (1+\ifk^{\f12-2\epsilon}),
\end{equation}
which together with $\ik^{2m}\lesssim \il^m\ikl^m$ and \eqref{eq3.2} implies
\begin{align*}
&\Big|\int_{2|k|<|k-l|} e^{2c\nu^\f13|k|^\f23 t}\ik^{2m}\ifk^{2\epsilon} \bigl(\int_\R \cM(k,D_y)\omega_k \cdot\p_y \phi_l \cdot(k-l)\omega_{k-l}dy \bigr)dkdl\Big| \\
 \lesssim &\big\|e^{c\nu^\f13|k|^\f23 t}(\ifk^{2\epsilon}+\ifk^\f12) \|\omega_k\|_{L^2_y}\big\|_{L^1_{k}} 
 \big\|e^{c\nu^\f13|l|^\f23 t}\il^m \ifl^\epsilon |l|\|\p_y \phi_l \|_{L^2_y}\big\|_{L^2_{l}} \\
&\qquad\qquad \times \big\|e^{c\nu^\f13|k-l|^\f23 t} \ikl^m \ifkl^\epsilon |k-l|^\f12\|\omega_{k-l}\|_{L^\oo_y}\big\|_{L^2_{k-l}} \\
\lesssim &\big\|e^{c\nu^\f13|k|^\f23 t}\ik^m\ifk^{\epsilon} \|\omega_k\|_{L^2_y}\big\|_{L^2_{k}}  \|(\ifk^\epsilon+\ifk^{\f12-\epsilon})\ik^{-m}\|_{L^2_k}
 \big\|e^{c\nu^\f13|l|^\f23 t}\il^m \ifl^\epsilon |l|\|\p_y \phi_l \|_{L^2_y}\big\|_{L^2_{l}} \\
&\qquad\qquad \times \big\|e^{c\nu^\f13|k-l|^\f23 t} \ikl^m \ifkl^\epsilon \|\nabla_{k-l}\omega_{k-l}\|_{L^2_y}\big\|_{L^2_{k-l}} \\
\lesssim & \nu^{-\f12}\|e^{c\nu^\f13|D_x|^\f23 t}\iDx^{m}\ifD^{\epsilon} \omega\|_{L^2} \|e^{c\nu^\f13|D_x|^\f23 t}\iDx^{m}\ifD^{\epsilon} \p_x\p_y\phi\|_{L^2}
  \\
&\qquad\qquad\times\nu^\f12\|e^{c\nu^\f13|D_x|^\f23 t}\iDx^{m}\ifD^{\epsilon} \nabla\omega\|_{L^2}.
\end{align*}

Combining the above estimates, we have proven
\begin{equation}\label{eq3.3}
\begin{aligned}
&\Big|\int_{\R^2} e^{2c\nu^\f13|k|^\f23 t}\ik^{2m}\ifk^{2\epsilon} \bigl(\int_\R \cM(k,D_y)\omega_k\cdot\p_y \phi_l \cdot(k-l)\omega_{k-l}dy \bigr)dkdl \Big|\\
\lesssim &  \nu^{-\f12}\|e^{c\nu^\f13|D_x|^\f23 t}\iDx^{m}\ifD^{\epsilon} \omega\|_{L^2}
\bigl( \|e^{c\nu^\f13|D_x|^\f23 t}\iDx^{m}\ifD^{\epsilon} \p_x\p_y\phi\|_{L^2}^2  \\
&\qquad+\nu\|e^{c\nu^\f13|D_x|^\f23 t}\iDx^{m}\ifD^{\epsilon} \nabla\omega\|_{L^2}^2 + \nu^\f13\|e^{c\nu^\f13|D_x|^\f23 t}\iDx^{m}\ifD^{\epsilon} |D_x|^\f13\omega\|_{L^2}^2\bigr).
\end{aligned}
\end{equation}

\begin{itemize}
\item For the term with $ l \phi_l \cdot \p_y\omega_{k-l}$.
\end{itemize}

When $\f{|k-l|}2\leq|k|\leq 2|k-l|$, we deduce from $\ik^m\ifk^\epsilon\lesssim \ikl^m\ifkl^\epsilon$ that
\begin{align*}
&\Big|\int_{\f{|k-l|}2\leq|k|\leq 2|k-l|} e^{2c\nu^\f13|k|^\f23 t}\ik^{2m}\ifk^{2\epsilon} \bigl(\int_\R \cM(k,D_y)\omega_k\cdot l \phi_l \cdot \p_y\omega_{k-l}dy \bigr)dkdl \Big| \\
 \lesssim &\big\|e^{c\nu^\f13|k|^\f23 t}\ik^{m}\ifk^{\epsilon} \|\omega_k\|_{L^2_y}\big\|_{L^2_{k}}
  \big\|e^{c\nu^\f13|l|^\f23 t}|l| \| \phi_l \|_{L^\oo_y}\big\|_{L^1_{l}}  \\
 &\qquad\qquad\times  \big\|e^{c\nu^\f13|k-l|^\f23 t}\ikl^{m}\ifkl^{\epsilon} \|\p_y\omega_{k-l}\|_{L^2_y}\big\|_{L^2_{k-l}} \\
\lesssim & \nu^{-\f12}\|e^{c\nu^\f13|D_x|^\f23 t}\iDx^{m}\ifD^{\epsilon} \omega\|_{L^2}
\|e^{c\nu^\f13|D_x|^\f23 t}\iDx^{m}\ifD^{\epsilon} \p_x\nabla \phi\|_{L^2}  \\
&\qquad\qquad\times\nu^\f12\|e^{c\nu^\f13|D_x|^\f23 t}\iDx^{m}\ifD^{\epsilon} \p_y\omega\|_{L^2},
\end{align*}
where we used from \eqref{eq3.2} that
\begin{align*}
&\|e^{c\nu^\f13|l|^\f23 t} |l|\| \phi_l \|_{L^\oo_y}\|_{L^1_{l}} \lesssim \|e^{c\nu^\f13|l|^\f23 t} |l|^{\f12}\|\nabla_l\phi_l \|_{L^2_y}\|_{L^1_{l}} \\
&\lesssim \|e^{c\nu^\f13|l|^\f23 t}  \il^m \ifl^\epsilon |l|\|\nabla_l\phi_l \|_{L^2_y}\|_{L^2_{l}}\||l|^{-\f12}\il^{-m}\ifl^{-\epsilon}\|_{L^2_l}
\lesssim \|e^{c\nu^\f13|D_x|^\f23 t}\iDx^{m}\ifD^{\epsilon} \p_x\nabla\phi\|_{L^2}.
\end{align*}

When $2|k-l|<|k|$, one can derive $\f{|k|}2\leq |l|\leq 2|k|$ and therefore $\ik^m\ifk^\epsilon\lesssim \il^m\ifl^\epsilon$, $|k-l|^\f12\lesssim |l|^\f12$. Then, we use Young's inequality and \eqref{eq3.2} to get
\begin{align*}
&\Big|\int_{2|k-l|<|k|} e^{2c\nu^\f13|k|^\f23 t}\ik^{2m}\ifk^{2\epsilon} \bigl(\int_\R \cM(k,D_y)\omega_k \cdot l \phi_l \cdot \p_y\omega_{k-l}dy \bigr)dkdl\Big| \\
 \lesssim &\big\|e^{c\nu^\f13|k|^\f23 t}\ik^{m}\ifk^{\epsilon} \|\omega_k\|_{L^2_y}\big\|_{L^2_{k}} \big\|e^{c\nu^\f13|l|^\f23 t}\il^m \ifl^\epsilon |l|^\f32\| \phi_l \|_{L^\oo_y}\big\|_{L^2_{l}}  \\
&\qquad\qquad\times \big\|e^{c\nu^\f13|k-l|^\f23 t} |k-l|^{-\f12}\|\p_y\omega_{k-l}\|_{L^2_y}\big\|_{L^1_{k-l}} \\
 \lesssim &\big\|e^{c\nu^\f13|k|^\f23 t}\ik^{m}\ifk^{\epsilon} \|\omega_k\|_{L^2_y}\big\|_{L^2_{k}} 
 \big\|e^{c\nu^\f13|l|^\f23 t}\il^m \ifl^\epsilon |l|\|\nabla_l \phi_l \|_{L^2_y}\big\|_{L^2_{l}}  \\
&\qquad\qquad\times \big\|e^{c\nu^\f13|k-l|^\f23 t}\ikl^m\ifkl^\epsilon \|\p_y\omega_{k-l}\|_{L^2_y}\big\|_{L^2_{k-l}} \||k-l|^{-\f12}\ikl^{-m}\ifkl^{-\epsilon}\|_{L^2_{k-l}}\\
\lesssim & \nu^{-\f12}\|e^{c\nu^\f13|D_x|^\f23 t}\iDx^{m}\ifD^{\epsilon} \omega\|_{L^2} \|e^{c\nu^\f13|D_x|^\f23 t}\iDx^{m}\ifD^{\epsilon} \p_x\nabla\phi\|_{L^2}
  \\
&\qquad\qquad\times\nu^\f12\|e^{c\nu^\f13|D_x|^\f23 t}\iDx^{m}\ifD^{\epsilon} \p_y\omega\|_{L^2}.
\end{align*}

The most difficult case is again that $2|k|<|k-l|$, and in such case $\f{|k-l|}2\leq |l|\leq 2|k-l|$, we have to use $L^2_k$ integrability to overcome the singularity of $\ifk^\epsilon$. Again, we can use $\ik^{2m}\lesssim \il^m\ikl^m$, \eqref{eq3.2} and \eqref{eq3.2a} to find 
\begin{align*}
&\Big|\int_{2|k|<|k-l|} e^{2c\nu^\f13|k|^\f23 t}\ik^{2m}\ifk^{2\epsilon} \bigl(\int_\R \cM(k,D_y)\omega_k \cdot l \phi_l \cdot \p_y\omega_{k-l}dy \bigr)dkdl \Big|\\
 \lesssim &\big\|e^{c\nu^\f13|k|^\f23 t}(\ifk^{2\epsilon}+\ifk^{\f12}) \|\omega_k\|_{L^2_y}\big\|_{L^1_{k}} 
 \big\|e^{c\nu^\f13|l|^\f23 t}\il^m \ifl^\epsilon |l|^\f32\| \phi_l \|_{L^\oo_y}\big\|_{L^2_{l}} \\
&\qquad\qquad \times \big\|e^{c\nu^\f13|k-l|^\f23 t} \ikl^m \ifkl^\epsilon \|\p_y\omega_{k-l}\|_{L^2_y}\big\|_{L^2_{k-l}} \\
\lesssim &\big\|e^{c\nu^\f13|k|^\f23 t}\ik^m\ifk^{\epsilon} \|\omega_k\|_{L^2_y}\big\|_{L^2_{k}}  \|(\ifk^\epsilon+\ifk^{\f12-\epsilon})\ik^{-m}\|_{L^2_k}
 \big\|e^{c\nu^\f13|l|^\f23 t}\il^m \ifl^\epsilon |l|\|\nabla_l \phi_l \|_{L^2_y}\big\|_{L^2_{l}} \\
&\qquad\qquad \times \big\|e^{c\nu^\f13|k-l|^\f23 t} \ikl^m \ifkl^\epsilon \|\p_y\omega_{k-l}\|_{L^2_y}\big\|_{L^2_{k-l}} \\
\lesssim & \nu^{-\f12}\|e^{c\nu^\f13|D_x|^\f23 t}\iDx^{m}\ifD^{\epsilon} \omega\|_{L^2} \|e^{c\nu^\f13|D_x|^\f23 t}\iDx^{m}\ifD^{\epsilon} \p_x\nabla\phi\|_{L^2}  \\
&\qquad\qquad\times\nu^\f12\|e^{c\nu^\f13|D_x|^\f23 t}\iDx^{m}\ifD^{\epsilon} \p_y\omega\|_{L^2}.
\end{align*}

Combining the above estimates, we have proven
\begin{equation}\label{eq3.4}
\begin{aligned}
&\Big|\int_{\R^2} e^{2c\nu^\f13|k|^\f23 t}\ik^{2m}\ifk^{2\epsilon} \bigl(\int_\R \cM(k,D_y)\omega_k \cdot l \phi_l \cdot \p_y\omega_{k-l}dy \bigr)dkdl\Big| \\
\lesssim &  \nu^{-\f12}\|e^{c\nu^\f13|D_x|^\f23 t}\iDx^{m}\ifD^{\epsilon} \omega\|_{L^2}
\bigl( \|e^{c\nu^\f13|D_x|^\f23 t}\iDx^{m}\ifD^{\epsilon} \p_x\nabla\phi\|_{L^2}^2  \\
&\qquad+\nu\|e^{c\nu^\f13|D_x|^\f23 t}\iDx^{m}\ifD^{\epsilon} \p_y\omega\|_{L^2}^2 \bigr).
\end{aligned}
\end{equation}

In all, we take \eqref{eq3.3} and \eqref{eq3.4} into \eqref{eq3.1}, and integrate the result inequality over $[0,t]$ to find
\begin{align*}
&\|\sqrt{\cM}e^{c\nu^\f13|D_x|^\f23 t}\iDx^{m}\ifD^{\epsilon} \omega(t)\|_{L^2}^2
+\nu\|e^{c\nu^\f13|D_x|^\f23 t}\iDx^{m}\ifD^{\epsilon} \nabla \omega\|_{L^2_t L^2}^2
 \\
&+\f{\nu^\f13}{16} \|e^{c\nu^\f13|D_x|^\f23 t}\iDx^{m}\ifD^{\epsilon} |D_x|^\f13\omega\|_{L^2_tL^2}^2
+\|e^{c\nu^\f13|D_x|^\f23 t}\iDx^{m}\ifD^{\epsilon} \p_x \nabla \phi\|_{L^2_tL^2}^2 \\
\leq & \|\sqrt{\cM}\iDx^{m}\ifD^{\epsilon} \omega_{\ini}\|_{L^2}^2+ C \nu^{-\f12}\|e^{c\nu^\f13|D_x|^\f23 t}\iDx^{m}\ifD^{\epsilon} \omega\|_{L^\oo_tL^2} \\
&\qquad \times \bigl(\nu\|e^{c\nu^\f13|D_x|^\f23 t}\iDx^{m}\ifD^{\epsilon} \nabla\omega\|_{L^2_tL^2}^2
+ \|e^{c\nu^\f13|D_x|^\f23 t}\iDx^{m}\ifD^{\epsilon} \p_x\nabla\phi\|_{L^2_tL^2}^2  \\
&\qquad\qquad\qquad + \nu^\f13\|e^{c\nu^\f13|D_x|^\f23 t}\iDx^{m}\ifD^{\epsilon} |D_x|^\f13\omega\|_{L^2_tL^2}^2\bigr),
\end{align*}
which together with $1\leq \cM\leq 1+2\pi$ finishes the proof of \eqref{eq:prop3.1}.
\end{proof}

Now, we are in the position to prove Theorem \ref{thm1}.
\begin{proof}[Proof of Theorem \ref{thm1}] 
Here, we fix a choice of $m$, $\epsilon$ and $c$ satisfying the conditions of Proposition \ref{prop3.1} and denote $C>1$ to be constant there. Let us define the following timespan:
\begin{equation}\label{def:T*}
T^*\eqdef \sup \{ T>0\ \big|\ \|e^{c\nu^\f13|D_x|^\f23 t}\iDx^{m}\ifD^{\epsilon} \omega(t)\|_{L^\oo_T L^2}\leq \f{\nu^\f12}{2C} \}.
\end{equation}
From \eqref{smallness1} with small enough $\e$, we know $T^*>0$.

By applying \eqref{eq:prop3.1}, we find that for any $T<T^*$, 
\begin{equation} \label{eq3.7}
\begin{aligned}
&\|e^{c\nu^\f13|D_x|^\f23 t}\iDx^{m}\ifD^{\epsilon} \omega(t)\|_{L^\oo_T L^2}^2
+\f12\bigl(\nu\|e^{c\nu^\f13|D_x|^\f23 t}\iDx^{m}\ifD^{\epsilon} \nabla \omega\|_{L^2_T L^2}^2 \\
&\quad+\nu^\f13 \|e^{c\nu^\f13|D_x|^\f23 t}\iDx^{m}\ifD^{\epsilon} |D_x|^\f13\omega\|_{L^2_TL^2}^2
+\|e^{c\nu^\f13|D_x|^\f23 t}\iDx^{m}\ifD^{\epsilon} \p_x \nabla \phi\|_{L^2_TL^2}^2\bigr) \\
&\qquad\qquad
\leq  C\|\iDx^{m}\ifD^{\epsilon} \omega_{\ini}\|_{L^2}^2\leq C\e^2\nu.
\end{aligned}
\end{equation}

By taking $\e<\f1{4C^2}$, the above inequality implies $\|e^{c\nu^\f13|D_x|^\f23 t}\iDx^{m}\ifD^{\epsilon} \omega(t)\|_{L^\oo_T L^2}\leq \f{\nu^\f12}{4C}$. Together with \eqref{def:T*}, a standard bootstrap argument will show that $T^*=+\oo$. \eqref{eq1.4} and \eqref{eq1.5} follows from \eqref{eq3.7} by taking $T=T^*=+\oo$.
\end{proof}

\section{Proof of Theorem \ref{thm2}}\label{section 4}

The goal of this section is to prove Theorem \ref{thm2}. 

As observed from the proof of Theorem \ref{thm1}, it seems impossible to propagate $\ifD^\epsilon$ for $\epsilon\geq \f12$. To make use of the better properties assumed in \eqref{smallness2} for initial data, we introduce the quasi-linearization and decompose the equation \eqref{eqs:w} into a linear part \eqref{eqs:Lw} and nonlinear part \eqref{eqs:NLw}. We will show that the linear part has better control of low horizontal frequencies with higher regularities, while the nonlinear part is $\nu^\f23$ small so that we can use some ideas in the proof of Theorem \ref{thm1}.

Inspired by the ideas in \cite{wei2023nonlinear}, the detailed estimates of the nonlinear part will be separated into short time scale $t\leq \nu^{-\f16}$ and long time scale $t\geq \nu^{-\f16}$. The critical time $\nu^{-\f16}$ is chosen due to the different methods to handle the transport terms. Since commutator estimates can overcome the loss of derivatives, the equations may behaves like the toy model $\p_t f =\nu^\f13 t f$, which maintains smallness only before $\nu^{-\f16}$. When $t\geq \nu^{-\f16}$, inviscid damping estimates will provide additional smallness $t^{-1}\leq \nu^\f16$ to help us control some difficulties with $\nu^\f12$ smallness.
\subsection{Decomposition}

Denote $\Lw$ by solving the following linear equation:
\begin{equation}\label{eqs:Lw}
\quad \left\{\begin{array}{l}
\displaystyle \pa_t \Lw-\nu\Delta \Lw +y\p_x \Lw =0, \qquad (t,x,y)\in\R^+\times\R^2, \\
\displaystyle  \Lw|_{t=0}=\omega_{\ini}(x,y).
\end{array}\right.
\end{equation}
The solution of the above equation can be represented by the Fourier methods:
\begin{equation}\label{eq4.2}
\hLw (t,k,\xi)=\hw_{\ini}(k,\xi+kt)e^{-\nu\int_0^t|k|^2+|\xi+k(t-s)|^2ds},
\end{equation}
and satisfies the following estimates:
\begin{lem}
{  For initial data satisfying \eqref{smallness2}, there exists a universal constant $c_0$ such that for any $t>0$, $\Lw$ satisfies 
\begin{equation}\label{eq:lem4.1}
\begin{aligned}
\|e^{c_0 \nu |D_x|^2 t^3}\iDt^6\ifD^4 \Lw\|_{L^\oo_tL^2}
+\|e^{c_0 \nu |D_x|^2 t^3}\iDt^5\ifD^4 \Lw\|_{L^\oo_tL^1}\\
+\nu^\f16\|e^{c_0 \nu |D_x|^2 t^3}\iDt^6\ifD^4 |D_x|^\f13\Lw\|_{L^2_tL^2} \lesssim \nu^{\f13+\delta}.
\end{aligned}
\end{equation}
}\end{lem}
\begin{proof}
In view of \eqref{eq4.2}, we can compute 
\begin{align*}
\int_0^t|k|^2+|\xi+k(t-s)|^2ds=(|k|^2+|\xi|^2)t+k\xi t^2+\f13 |k|^2 t^3 
=|k|^2 t+\f{|k|^2t^3}{12}+|\xi+\f{k t}2 |^2t.
\end{align*}
It is easy to obverse that both $e^{-\nu|D_x|^2t}$ and $e^{-\nu|D_x+\f{ t}2 D_y|^2 t}$ are uniformly bounded on $L^p(\R^2)$ for all $t\geq 0$ and $1\leq p\leq +\oo$. This proves \eqref{eq:lem4.1} for $c_0<\f1{12}$. 
\end{proof}

\begin{rmk}
{ 
When we do the estimates, one can use that for any $c\geq 0$, there holds
 \begin{equation}\notag
\begin{aligned}
&\|e^{c \nu^\f13 |D_x|^{\f23} t}\iDt^6\ifD^4 \Lw\|_{L^\oo_tL^2}\\
&\qquad+\nu^\f16\|e^{c \nu^\f13 |D_x|^\f23 t}\iDt^6\ifD^4 |D_x|^\f13\Lw\|_{L^2_tL^2} \lesssim \nu^{\f13+\delta}.
\end{aligned}
\end{equation}
}\end{rmk}

Let us denote $\NLw=\omega-\Lw$ which solves:
\begin{equation}\label{eqs:NLw}
\quad \left\{\begin{array}{l}
\displaystyle \pa_t \NLw-\nu\Delta \NLw +y\p_x \NLw +(\Lu+\NLu)\cdot \nabla (\Lw+\NLw)=0, \\
\displaystyle \Lu =(-\p_y, \p_x)(-\Delta)^{-1} \Lw, \quad \NLu =(-\p_y, \p_x)(-\Delta)^{-1} \NLw, \\
\displaystyle  \NLw|_{t=0}=0.
\end{array}\right.
\end{equation}

The remaining part of this proof is to estimate the nonlinear part $\NLw$. 

\subsection{Estimates of $\NLw$ in short time scale $t\leq \nu^{-\f16}$} For time scale before $\nu^{-\f16}$, we will use transport structure and commutator estimates to overcome the loss of derivatives. However, the commutators between functions and the multiplier $\ifD^\epsilon$ will cause more singularities. For such reason, we first make a standard $H^3$ energy estimate, then do the estimate of $W^{2,p}$ norm for some $1<p<2$ with the help of $H^3$ control, and finally we will use $W^{2,p}$ norm to get the $\ifD^\epsilon$ regularity.

\begin{prop}\label{prop4.1}
{  For any $1<p<2$, if the initial data $\omega_{\ini}$ satisfies \eqref{smallness2}, then one has
\begin{equation}\label{eq:prop4.1}
\begin{aligned}
&\f{d}{dt}\|\iDt^3 \NLw\|_{L^2}\lesssim \langle t\rangle^{-2}\nu^{\f23+2\delta}\\
&+\langle t\rangle^{1+\delta}\bigl(\nu^{\f13+\delta}+\|\iDt^3 \NLw\|_{L^2}\bigl)\bigl(\|\NLw\|_{L^p}+ \|\iDt^3 \NLw\|_{L^2}\bigr).
\end{aligned}
\end{equation}
}\end{prop}
\begin{proof}
We take the inner product of \eqref{eqs:NLw} with $\iDt^{6}\NLw$ to find that
\begin{equation}\label{eq4.5}
\begin{aligned}
&\f12\f{d}{dt}\|\iDt^3 \NLw\|_{L^2}^2+\nu \|\iDt^3 \nabla \NLw\|_{L^2}^2 \\
&\qquad=-\Re \Bigl( \iDt^3 \big[(\Lu+\NLu)\cdot\nabla (\Lw+\NLw)\big] \Big| \iDt^3 \NLw \Bigr).
\end{aligned}\end{equation}

First, we consider the source term $\Lu\cdot \nabla \Lw$, which can be written by the Fourier transform as
\begin{align*}
&\Big|\bigl( \iDt^3 (\Lu\cdot\nabla \Lw) \big| \iDt^3 \NLw \bigr)\Big| \\
=& \Big|\int_{\R^4} \ikt^6 \hNLw_k(\xi) \f{\eta(k-l)-l(\xi-\eta)}{|l|^2+|\eta|^2}\hLw_l(\eta)\hLw_{k-l}(\xi-\eta)dkdld\xi d\eta\Big|.
\end{align*}
Using the following trick
\begin{equation}\label{change of variables}
\begin{aligned}
\eta(k-l)-l(\xi-\eta)&=(\eta+lt)(k-l)-l(\xi-\eta+(k-l)t)\\
&\leq \ilt \iklt
\end{aligned}
\end{equation}
and the point-wise inviscid damping estimate:
\begin{equation}\label{pointwise inviscid damping}
\f1{|l|^2+|\eta|^2}\lesssim \f{\ilt^2}{|l|^4 \langle t\rangle^2}\lesssim \langle t\rangle^{-2} \ifl^4\ilt^2,
\end{equation}
one can use $\ikt^3\lesssim \ilt^3\iklt^3$ to prove
\begin{equation}\label{source term:LL}
\begin{aligned}
&\Big|\bigl( \iDt^3 (\Lu\cdot\nabla \Lw) \big| \iDt^3 \NLw \bigr)\Big| \\
\lesssim & \langle t \rangle^{-2} \|\ikt^3 \hNLw_k(\xi)\|_{L^2_{k,\xi}}
\|\ilt^6 \ifl^4 \hLw_l(\eta)\|_{L^2_{l,\eta}}
 \\
&\qquad \times \|\iklt^{-2}\|_{L^2_{k-l,\eta}}\|\iklt^6\hLw_{k-l}(\xi-\eta)\|_{L^2_{k-l,\xi}}\\
\lesssim & \langle t\rangle^{-2}\|\iDt^3 \NLw\|_{L^2}\|\iDt^6\ifD^4 \Lw\|_{L^2}\\
&\qquad\qquad \times\|\iDt^6 \Lw\|_{L^2}.
\end{aligned}
\end{equation}

For the reaction term $\NLu\cdot\nabla\Lw$, one can write by Fourier transformation that
\begin{align*}
&\Big|\bigl( \iDt^3 (\NLu\cdot\nabla \Lw) \big| \iDt^3 \NLw \bigr)\Big| \\
=& \Big|\int_{\R^4} \ikt^6 \hNLw_k(\xi) \f{\eta(k-l)-l(\xi-\eta)}{|l|^2+|\eta|^2}\hNLw_l(\eta)\hLw_{k-l}(\xi-\eta)dkdld\xi d\eta\Big|.
\end{align*}
When $|l|^2+|\eta|^2\geq1$, one can use $\ikt^3\lesssim \ilt^3\iklt^3$ and
\begin{equation}\label{eq4.6}
|k-l|+|\xi-\eta|\leq |k-l|+|\xi-\eta+(k-l)t|+|k-l|t\lesssim \langle t\rangle\iklt
\end{equation}
to prove 
\begin{align*}
&\Big|\int_{|l|^2+|\eta|^2\geq1} \ikt^6 \hNLw_k(\xi) \f{\eta(k-l)-l(\xi-\eta)}{|l|^2+|\eta|^2}\hNLw_l(\eta)\hLw_{k-l}(\xi-\eta)dkdld\xi d\eta\Big|\\
\lesssim & \|\ikt^3 \hNLw_k(\xi)\|_{L^2_{k,\xi}}\|\ilt^3\hNLw_l(\eta)\|_{L^2_{l,\eta}} \langle t\rangle \|\iklt^{-2}\|_{L^2_{k-l,\xi-\eta}} \\
&\qquad\qquad\times\|\iklt^6 \hLw_{k-l}(\xi-\eta)\|_{L^2_{k-l,\xi-\eta}}\\
\lesssim & \langle t\rangle\|\iDt^3 \NLw\|_{L^2}^2\|\iDt^6 \Lw\|_{L^2}.
\end{align*}
When $|l|^2+|\eta|^2<1$, to overcome some singularity near $(l,\eta)=(0,0)$, we use a more detailed estimate of $\ikt^3$ as 
$$
\ikt^3\lesssim \iklt^3+\ilt^2(|l|+|\eta+lt|)$$
together with 
\begin{equation}\label{eq4.7}
|l|+|\eta+lt|\leq \ilt^{1-\delta}\langle t\rangle^\delta (|l|+|\eta|)^\delta
\end{equation}
to derive
\begin{align*}
&\Big|\int_{|l|^2+|\eta|^2<1} \ikt^6 \hNLw_k(\xi) \f{\eta(k-l)-l(\xi-\eta)}{|l|^2+|\eta|^2}\hNLw_l(\eta)\hLw_{k-l}(\xi-\eta)dkdld\xi d\eta\Big|\\
\lesssim & \langle t\rangle\|\ikt^3 \hNLw_k(\xi)\|_{L^2_{k,\xi}}\bigl(\|\hNLw_l(\eta)\|_{L^{p'}_{l,\eta}}\|(|l|+|\eta|)^{-1}\chi|_{|l|^2+|\eta|^2<1}\|_{L^p_{l,\eta}}  \\
&\quad\times \|\iklt^4 \hLw_{k-l}(\xi-\eta)\|_{L^2_{k-l,\xi-\eta}}
+\langle t\rangle^\delta\|\ilt^{3-\delta}\hNLw_l(\eta)\|_{L^2_{l,\eta}}\\
&\qquad\times \|(|l|+|\eta|)^{\delta-1}\chi|_{|l|^2+|\eta|^2<1}\|_{L^2_{l,\eta}}\|\iklt \hLw_{k-l}(\xi-\eta)\|_{L^2_{k-l,\xi-\eta}}\bigr)\\
\lesssim & \langle t\rangle^{1+\delta}\|\iDt^3 \NLw\|_{L^2}\bigl(\|\NLw\|_{L^p}+\|\iDt^3\NLw\|_{L^2}\bigr)\\
&\qquad\qquad\times\|\iDt^4 \Lw\|_{L^2}.
\end{align*}
Combining the above two cases, we summarize
\begin{equation}\label{eq4.8}
\begin{aligned}
&\Big|\bigl( \iDt^3 (\NLu\cdot\nabla \Lw) \big| \iDt^3 \NLw \bigr)\Big| \\
\lesssim & \langle t\rangle^{1+\delta}\|\iDt^3 \NLw\|_{L^2}\bigl(\|\NLw\|_{L^p}+\|\iDt^3\NLw\|_{L^2}\bigr)\\
&\qquad\qquad\times\|\iDt^6 \Lw\|_{L^2}.
\end{aligned}
\end{equation}

For nonlinear terms with derivatives on $\NLw$, we can use the following divergence-free structure:
$$
\bigl( u\cdot \nabla \iDt^3 \NLw \big| \iDt^3\NLw\bigr)=0
$$
to rewrite
\begin{align*}
&\Big|\bigl( \iDt^3 (u\cdot\nabla \NLw) \big| \iDt^3 \NLw \bigr)\Big| \\
=& \Big|\int_{\R^4} \ikt^3 \hNLw_k(\xi)\bigl(\ikt^3-\iklt^3\bigr)\\
&\qquad\qquad\times \f{\eta(k-l)-l(\xi-\eta)}{|l|^2+|\eta|^2}\hw_l(\eta)\hNLw_{k-l}(\xi-\eta)dkdld\xi d\eta\Big|.
\end{align*}
For the commutator, one has the following classical estimate:
$$
|\ikt^3-\iklt^3|\lesssim (|l|+|\eta+lt|)(\ilt^2+\iklt^2),
$$
which together with \eqref{eq4.6} and \eqref{eq4.7} implies that for $|l|+|\eta|<1$,
\begin{align*}
&\Big|\int_{|l|+|\eta|<1} \ikt^3 \hNLw_k(\xi)\bigl(\ikt^3-\iklt^3\bigr)\\
&\qquad\qquad\times \f{\eta(k-l)-l(\xi-\eta)}{|l|^2+|\eta|^2}\hw_l(\eta)\hNLw_{k-l}(\xi-\eta)dkdld\xi d\eta\Big|\\
\lesssim & \|\ikt^3\hNLw_k(\xi)\|_{L^2_{k,\xi}}
\langle t\rangle^\delta \|\ilt^{3-\delta}\hw_l(\eta)\|_{L^2_{l,\eta}}\|(|l|+|\eta|)^{\delta-1}\chi|_{|l|+|\eta|<1}\|_{L^2_{l,\eta}}\\
&\qquad\qquad\times \langle t\rangle \|\iklt^3\hNLw_{k-l}(\xi-\eta)\|_{L^2_{k-l,\xi-\eta}}\\
\lesssim & \langle t\rangle^{1+\delta} \|\iDt^3 \NLw\|_{L^2}^2 \|\iDt^3 \omega\|_{L^2},
\end{align*}
and for $|l|+|\eta|\geq 1$,
\begin{align*}
&\Big|\int_{|l|+|\eta|\geq1} \ikt^3 \hNLw_k(\xi)\bigl(\ikt^3-\iklt^3\bigr)\\
&\qquad\qquad\times \f{\eta(k-l)-l(\xi-\eta)}{|l|^2+|\eta|^2}\hw_l(\eta)\hNLw_{k-l}(\xi-\eta)dkdld\xi d\eta\Big|\\
\lesssim & \|\ikt^3\hNLw_k(\xi)\|_{L^2_{k,\xi}}
\|\ilt^{3}\hw_l(\eta)\|_{L^2_{l,\eta}}\\
&\qquad\qquad\times \bigl(\|\ilt^{-2}\|_{L^2_{l,\eta}}+\|\iklt^{-2}\|_{L^2_{k-l,\xi-\eta}}\bigr)\\
&\qquad\qquad\times\langle t\rangle \|\iklt^3\hNLw_{k-l}(\xi-\eta)\|_{L^2_{k-l,\xi-\eta}}\\
\lesssim & \langle t\rangle \|\iDt^3 \NLw\|_{L^2}^2 \|\iDt^3 \omega\|_{L^2}.
\end{align*}
Combining the above two cases, we summarize
\begin{equation}\label{eq4.9}
\begin{aligned}
&\Big|\bigl( \iDt^3 (u\cdot\nabla \NLw ) \big| \iDt^3 \NLw \bigr)\Big| \\
\lesssim & \langle t\rangle^{1+\delta} \|\iDt^3 \NLw\|_{L^2}^2 \|\iDt^3 \omega\|_{L^2}.
\end{aligned}
\end{equation}

In all, we bring \eqref{source term:LL}, \eqref{eq4.8} and \eqref{eq4.9} into \eqref{eq4.5} and arrive at
\begin{equation}\notag
\begin{aligned}
&\f{d}{dt}\|\iDt^3 \NLw\|_{L^2}^2\lesssim \|\iDt^3 \NLw\|_{L^2} \Bigl(\langle t\rangle^{-2}\|\iDt^6\ifD^4 \Lw\|_{L^2}^2 \\
&\qquad+\langle t\rangle^{1+\delta}\bigl(\|\iDt^6\Lw\|_{L^2} +\|\iDt^3 \NLw\|_{L^2}\bigr)\\
&\qquad\qquad\times\bigl(\|\NLw\|_{L^p}+ \|\iDt^3 \NLw\|_{L^2}\bigr)\Bigr),
\end{aligned}
\end{equation}
which together with \eqref{eq:lem4.1} finishes the proof of Proposition \ref{prop4.1}.
\end{proof}

\begin{prop}\label{prop4.2}
{  For any $1<p<2$, if the initial data $\omega_{\ini}$ satisfies \eqref{smallness2}, then one has
\begin{equation}\label{eq:prop4.2}
\begin{aligned}
&\f{d}{dt}\|\iDt^2 \NLw\|_{L^p}\lesssim \langle t\rangle^{-2}\nu^{\f23+2\delta}+\langle t\rangle \bigl(\nu^{\f13+\delta}+\|\iDt^2 \NLw\|_{L^{p}}\bigr)\\
&\qquad\qquad\qquad\times\bigl(\|\iDt^2 \NLw\|_{L^{p}}+\|\iDt^3\NLw\|_{L^2}\bigr).
\end{aligned}
\end{equation}
}\end{prop}
\begin{proof}
For the $L^p$ energy estimate, we apply $\iDt^2$ to \eqref{eqs:NLw} and take the inner product of the resulting equation with $|\iDt^2 \NLw|^{p-2}\iDt^2\NLw$ to find that
\begin{equation}\label{eq4.15}
\begin{aligned}
&\f1p\f{d}{dt}\|\iDt^2 \NLw\|_{L^p}^p\\
\leq&-\Re \bigl( \iDt^2(u\cdot \nabla \omega) \big| |\iDt^2 \NLw|^{p-2}\iDt^2\NLw \bigr),
\end{aligned}
\end{equation}
where we used from the commutativity between $\iDt^2$ and $\p_t+y\p_x$ that
\begin{align*}
&\Re \big( \iDt^2 (\p_t+y\p_x)\NLw \big||\iDt^2 \NLw|^{p-2}\iDt^2\NLw \bigr)\\
=& \f1p\int_{\R^2} (\p_t+y\p_x)|\iDt^2 \NLw|^{p}dxdy =\f1p \f{d}{dt}\|\iDt^2 \NLw\|_{L^p}^p,
\end{align*}
and 
\begin{align*}
&-\Re \big( \iDt^2 \p_y^2\NLw \big||\iDt^2 \NLw|^{p-2}\iDt^2\NLw \bigr)\\
=& (p-1)\int_{\R^2} |\iDt^2 \NLw|^{p-2} |\iDt^2 \p_y \NLw|^2dxdy \geq 0.
\end{align*}
By applying H\"older inequality, we have
\begin{align*}
&|\Re \bigl( \iDt^2(u\cdot \nabla \omega) \big| |\iDt^2 \NLw|^{p-2}\iDt^2\NLw \bigr)| \\
\leq &\| \iDt^2(u\cdot \nabla \omega)\|_{L^p} \|\iDt^2 \NLw\|_{L^p}^{p-1},
\end{align*}
and therefore \eqref{eq4.15} implies
$$
\f{d}{dt}\|\iDt^2 \NLw\|_{L^p}\leq \|\iDt^2(u\cdot \nabla \omega)\|_{L^p}.
$$
Recalling that $u\cdot \nabla \omega=(\Lu+\NLu)\cdot \nabla (\Lw+\NLw)$, we shall consider those terms one by one.

For the source term $\Lu\cdot \nabla \Lw$, one can follow the proof of \eqref{source term:LL} to get
\begin{align*}
&\|\iDt^2(\Lu\cdot \nabla \Lw)\|_{L^p} \\
\lesssim &\|\iDt^3 \Delta^{-1} \Lw\|_{L^2}\|\iDt^3 \Lw\|_{L^{\f{2p}{2-p}}} \\
\lesssim &\langle t \rangle^{-2}\|\iDt^5\ifD^4 \Lw\|_{L^2}\|\iDt^5 \Lw\|_{L^{2}},
\end{align*}
where we used that for $1<p<2$,
\begin{align*}
&\|\iDt^3 \Lw\|_{L^{\f{2p}{2-p}}}\lesssim \|\ikt^3\hLw_k(\xi)\|_{L^{\f{2p}{3p-2}}_{k,\xi}}\\
\lesssim &\|\ikt^5 \hLw_k(\xi)\|_{L^2_{k,\xi}}\|\ikt^{-2}\|_{L^{p'}_{k,\xi}}\lesssim \|\iDt^5 \Lw\|_{L^2}.
\end{align*}

For the other terms, we can use the following classical estimate of Biot-Savart law $$\|\nabla^\perp(-\Delta)^{-1}f\|_{L^{\f{2p}{2-p}}}\lesssim \|f\|_{L^p}$$
and \eqref{eq4.6} to prove that
\begin{align*}
&\|\iDt^2(\NLu\cdot \nabla \Lw+u\cdot \nabla \NLw)\|_{L^p}\\
\lesssim &\|\iDt^2 \NLu\|_{L^{\f{2p}{2-p}}}\|\iDt^2 \nabla\Lw\|_{L^2}\\
&\qquad\qquad+\|\iDt^2 u\|_{L^{\f{2p}{2-p}}}\|\iDt^2 \nabla\NLw\|_{L^2}\\
\lesssim &\langle t\rangle\Bigl(\|\iDt^2 \NLw\|_{L^{p}}\|\iDt^3\Lw\|_{L^2}\\
&\qquad+\bigl(\|\iDt^2 \Lw\|_{L^{p}}+\|\iDt^2 \NLw\|_{L^{p}}\bigr)\|\iDt^3\NLw\|_{L^2}\Bigr).
\end{align*}

In all, we have proven
\begin{align*}
&\f{d}{dt}\|\iDt^2 \NLw\|_{L^p}\lesssim \langle t\rangle^{-2}\|\iDt^5\ifD^4 \Lw\|_{L^2}^2\\
&+\langle t\rangle \bigl(\|\iDt^2 \Lw\|_{L^{p}}+\|\iDt^3\Lw\|_{L^2}+\|\iDt^2 \NLw\|_{L^{p}}\bigr)\\
&\qquad\qquad\times\bigl(\|\iDt^2 \NLw\|_{L^{p}}+\|\iDt^3\NLw\|_{L^2}\bigr),
\end{align*}
which together with \eqref{eq:lem4.1} finishes the proof of \eqref{eq:prop4.2}. 
\end{proof}

Combining the above two propositions, we can close the estimates for $t<\nu^{-\f16}$ as follows:
\begin{prop}\label{prop4.3}
{ 
Let $1<p<2$, if the initial data $\omega_{\ini}$ satisfies \eqref{smallness2} with $\nu$ smaller than some universal constant, then for any $T\leq \nu^{-\f16}$, one has
\begin{equation}\label{eq:prop4.3}
\|\iDt^3\NLw\|_{L^\oo_T L^2}+\|\iDt^2\NLw\|_{L^\oo_TL^p}\lesssim \nu^{\f23+2\delta}.
\end{equation}
}
\end{prop}
\begin{proof}
Fix some $1<p<2$. Let us denote 
$$\Omega(t)=\|\iDt^3\NLw(t)\|_{L^2}+\|\iDt^2\NLw(t)\|_{L^p}.$$
By adding up \eqref{eq:prop4.1} and \eqref{eq:prop4.2}, we conclude that 
$$
\f{d}{dt}\Omega(t)\leq C\langle t\rangle^{-2}\nu^{\f23+2\delta}+ C\langle t\rangle^{1+\delta}\bigl(\nu^{\f13+\delta}+\Omega(t)\bigr)\Omega(t).
$$
By integrating the above inequality over $[0,T]$, we arrive at
\begin{equation}\label{eq4.14}
\sup_{t\in[0,T]}\Omega(t)\leq C\nu^{\f23+2\delta}+\f{C}{2+\delta} \langle T\rangle^{2+\delta} \bigl(\nu^{\f13+\delta}+\sup_{t\in[0,T]}\Omega(t)\bigr)\sup_{t\in[0,T]}\Omega(t).
\end{equation}
For $\langle T\rangle\leq \nu^{-\f16}$, by assuming $\nu$ is small enough such that
$$
\langle T\rangle^{2+\delta}\nu^{\f13+\delta}\leq \nu^{\f56\delta}\leq \f{2+\delta}{4C} \andf  \nu^{\f13+\delta}\leq \f{1}{2C},
$$
we can use a standard bootstrap argument to derive from \eqref{eq4.14} that
$$
\sup_{t\in[0,T]}\Omega(t)\leq 2C\nu^{\f23+2\delta},
$$
which finishes the proof.
\end{proof}

With \eqref{eq:prop4.3}, we can control the following norms: 

\begin{col}\label{col4.1}
{  Let $0<\epsilon<\f12$ and $c>0$, if the initial data $\omega_{\ini}$ satisfies \eqref{smallness2} with $\nu$ smaller than some universal constant, then for any $T\leq \nu^{-\f16}$, one has
\begin{equation}\label{eq:col4.1}
\|e^{c\nu^\f13\lambda(D_x)t}\iDx \ifD^\epsilon \NLw\|_{L^\oo_T L^2}\lesssim \nu^{\f23+2\delta}.
\end{equation}
}\end{col}
\begin{proof}
From \eqref{def:lambda}, we know $\lambda(k)\leq 1$, therefore, for $t\leq \nu^{-\f16}$, one deduce from $\nu^{\f13}t\leq \nu^\f13\leq 1$ that
$$
\|e^{c\nu^\f13\lambda(D_x)t}\iDx \ifD^\epsilon \NLw\|_{L^2}\lesssim \|\iDx \ifD^\epsilon \NLw\|_{L^2}.
$$
For $\epsilon<\f12$, if $1<p<2$ is so close to $1$ that $\f1p>\epsilon+\f12$, then one has 
$$
\|\ifk^\epsilon \ikt^{-1}\|_{L^{\f{2p}{2-p}}_{k,\xi}} \lesssim 1,
$$
and therefore,
\begin{align*}
\|\iDx \ifD^\epsilon \NLw\|_{L^2}&\leq \|\ikt^2 \hNLw_k(\xi)\|_{L^{p'}_{k,\xi}}\|\ifk^\epsilon \ikt^{-1}\|_{L^{\f{2p}{2-p}}_{k,\xi}}\\
&\lesssim \|\iDt^2 \NLw\|_{L^p}.
\end{align*}
This together with \eqref{eq:prop4.3} finishes the proof.
\end{proof}

\begin{rmk}
{ 
In this work, we only use $\ifD^\epsilon$ to prevent the frequencies from concentrating near zero. However, our conclusion here is stronger: see from the proof of Corollary \ref{col4.1}, the $W^{2,p}$ norms can provide control for all frequencies.
}
\end{rmk}

\subsection{Estimates of $\NLw$ in long time scale $t\geq \nu^{-\f16}$}
From the last subsection, we already estimate the solution to some $T_0=\nu^{-\f16}$. For $t>T_0$, we shall do energy estimates with the multiplier $\ifD^\epsilon$ and therefore we cannot use any commutator estimates.

In this section, we will always denote $\cM=\cM_1+\cM_2+\cM_3+1$.

\begin{prop} \label{prop4.4}
{  Let $\f{1-\delta}2<\epsilon<\f12$, $0<\kappa<\f{2\epsilon}{1-\delta}-1$ and $c<c(\delta,\epsilon)$ small enough, if the initial data $\omega_{\ini}$ satisfies \eqref{smallness2} with $\nu\leq 1$, then for any $\nu^{-\f16}\leq T_0<T$, one has
\begin{equation}\label{eq:prop4.4}
\begin{aligned}
&\Bigl(1-C\nu^{\f16+\delta}\Bigr)\|e^{c\nu^\f13\lambda(D_x) t}\iDx\ifD^{\epsilon} \NLw\|_{L^\oo_{[T_0,T]}L^2}^2 \\
&+\Bigl(1-C\nu^{\f56\delta}-C\nu^{-\f12}\|e^{c\nu^\f13\lambda(D_x) t}\iDx\ifD^{\epsilon} \NLw\|_{L^\oo_{[T_0,T]}L^2}\Bigr) \\
&\quad\times\Bigl(\nu\|e^{c\nu^\f13\lambda(D_x) t}\iDx\ifD^{\epsilon} \nabla \NLw\|_{L^2_{[T_0,T]}L^2}^2+\nu^\f13 \|e^{c\nu^\f13\lambda(D_x) t}\iDx\ifD^{\epsilon} |D_x|^\f13\NLw\|_{L^2_{[T_0,T]}L^2}^2
\\
&\qquad+\|e^{c\nu^\f13\lambda(D_x) t}\iDx\ifD^{\epsilon} \p_x \nabla \NLp\|_{L^2_{[T_0,T]}L^2}^2+\|e^{c\nu^\f13\lambda(D_x) t}\iDx\ifD^{\epsilon} \sqrt{\Upsilon} \NLw\|_{L^2_{[T_0,T]}L^2}^2\Bigr)\\
&\leq C\|e^{c\nu^\f13\lambda(D_x) T_0}\iDx\ifD^{\epsilon} \NLw(T_0)\|_{L^2}^2+C\nu^{\f56+2\delta}\|e^{c\nu^\f13\lambda(D_x) t}\iDx\ifD^\epsilon\NLw\|_{L^\oo_{[T_0,T]}L^2}.
\end{aligned}
\end{equation}
}\end{prop}

\begin{proof}In this proof, we fix the choice of $\kappa \in (0,\f{2\epsilon}{1-\delta}-1 )$ in $\cM_3$.

By taking the inner product of \eqref{eqs:NLw} with $\cM e^{2c\nu^\f13 \lambda(D_x) t}\iDx^{2}\ifD^{2\epsilon} \NLw$, we infer that\begin{align*}
&\f{d}{dt}\|\sqrt{\cM}e^{c\nu^\f13\lambda(D_x) t}\iDx^{m}\ifD^{\epsilon} \NLw\|_{L^2}^2
-2c\nu^\f13 \|\sqrt{\cM}e^{c\nu^\f13\lambda(D_x) t}\iDx\ifD^{\epsilon}\sqrt{\lambda(D_x)}\NLw\|_{L^2}^2 \\
&\qquad\qquad+2\nu\|\sqrt{\cM}e^{c\nu^\f13\lambda(D_x) t}\iDx\ifD^{\epsilon} \nabla \NLw\|_{L^2}^2 \\
&\qquad\qquad+\int_{\R^2} (-\p_t+k\p_\xi)\cM(k,\xi) e^{2c\nu^\f13\lambda(k) t}\langle k\rangle^{2}\ifk^{2\epsilon} |\hNLw(k,\xi)|^2 dkd\xi \\
&\qquad=-2\Re \bigl( (\Lu+\NLu)\cdot\nabla (\Lw+\NLw) \big| \cM e^{2c\nu^\f13\lambda(D_x) t}\iDx^{2}\ifD^{2\epsilon} \NLw \bigr).
\end{align*}
Thanks to \eqref{eq:M1}, \eqref{eq:M2} and \eqref{eq:M3}, we use the fact that $1\leq \cM(k,\xi)\leq C_\kappa$ and $\lambda(k)\leq |k|^\f23$ to deduce that for $c<\f1{16C_\kappa}$, we have 
\begin{equation}\label{eq4.18}
\begin{aligned}
&\f{d}{dt}\|\sqrt{\cM}e^{c\nu^\f13\lambda(D_x) t}\iDx\ifD^{\epsilon} \NLw\|_{L^2}^2
+\nu\|e^{c\nu^\f13\lambda(D_x) t}\iDx\ifD^{\epsilon} \nabla \NLw\|_{L^2}^2
 \\
&\quad+\f{\nu^\f13}{16} \|e^{c\nu^\f13\lambda(D_x) t}\iDx\ifD^{\epsilon} |D_x|^\f13\NLw\|_{L^2}^2
+\|e^{c\nu^\f13\lambda(D_x) t}\iDx\ifD^{\epsilon} \p_x \nabla \NLp\|_{L^2}^2 \\
&\qquad+\|e^{c\nu^\f13\lambda(D_x) t}\iDx\ifD^{\epsilon} \sqrt{\Upsilon(t,D_x,D_y)} \NLw\|_{L^2}^2\\
&\leq 2\Big|\Re \bigl( (\Lu+\NLu)\cdot\nabla (\Lw+\NLw) \big| \cM e^{2c\nu^\f13\lambda(D_x) t}\iDx^{2}\ifD^{2\epsilon} \NLw \bigr)\Big|,
\end{aligned}
\end{equation}
where $\NLp=\Delta^{-1}\NLw$ denotes the stream function.

\begin{itemize}
\item For the treatment of the source term $\Lu\cdot\nabla \Lw$.
\end{itemize}

By applying the Fourier transformation, we write
\begin{align*}
&\Big|\bigl( \Lu\cdot\nabla \Lw \big| \cM e^{2c\nu^\f13\lambda(D_x) t}\iDx^{2}\ifD^{2\epsilon} \NLw \bigr)\Big|\\
=&\Big|\int_{\R^4} \cM(t,k,\xi)e^{2c\nu^\f13\lambda(k)t}\ik^2 \ifk^{2\epsilon} \hNLw_k(\xi) \f{\eta(k-l)-l(\xi-\eta)}{|l|^2+|\eta|^2} \hLw_l(\eta) \hLw_{k-l}(\xi-\eta) dkdld\xi d\eta \Big|.
\end{align*}

When $|k|\geq 1$, similarly to \eqref{source term:LL}, we can use \eqref{change of variables} and \eqref{pointwise inviscid damping} to get
\begin{align*}
&\Big|\int_{|k|\geq 1} \cM(t,k,\xi)e^{2c\nu^\f13\lambda(k)t}\ik^2 \ifk^{2\epsilon} \hNLw_k(\xi) \f{\eta(k-l)-l(\xi-\eta)}{|l|^2+|\eta|^2} \hLw_l(\eta) \hLw_{k-l}(\xi-\eta) dkdld\xi d\eta  \Big| \\
\lesssim & \|e^{c\nu^\f13\lambda(k)t}\ik\hNLw_k(\xi)\|_{L^2_{k,\xi}}\langle t\rangle^{-2}\|e^{c\nu^\f13\lambda(l)t}\ilt^4\ifl^4\hLw_l(\eta)\|_{L^2_{l,\eta}}\\
&\qquad\times\|\iklt^{-2}\|_{L^2_{k-l,\xi-\eta}}\\
&\qquad\qquad \times\|e^{c\nu^\f13\lambda(k-l)t}\iklt^4\hLw_{k-l}(\xi-\eta)\|_{L^2_{k-l,\xi-\eta}}\\
\lesssim & \langle t\rangle^{-2}\|e^{c\nu^\f13\lambda(D_x)t}\iDx\NLw\|_{L^2} \|e^{c\nu^\f13\lambda(D_x)t}\iDx^4\ifD^4\Lw\|_{L^2}^2
\end{align*}

When $|k|<1$, we can use the fact $\ifk^\epsilon\in L^2_k$ to derive similarly
\begin{align*}
&\Big|\int_{|k|< 1} \cM(t,k,\xi)e^{2c\nu^\f13\lambda(k)t}\ik^2 \ifk^{2\epsilon} \hNLw_k(\xi) \f{\eta(k-l)-l(\xi-\eta)}{|l|^2+|\eta|^2} \hLw_l(\eta) \hLw_{k-l}(\xi-\eta) dkdld\xi d\eta \Big| \\
\lesssim & \|\ifk^\epsilon \|_{L^2_k([-1,1])}\|e^{c\nu^\f13\lambda(k)t}\ik\ifk^\epsilon\hNLw_k(\xi)\|_{L^2_{k,\xi}}\langle t\rangle^{-2}\|e^{c\nu^\f13\lambda(l)t}\ilt^4\ifl^4\hLw_l(\eta)\|_{L^2_{l,\eta}}\\
&\times\|\iklt^{-2}\|_{L^2_{\xi-\eta}}\|e^{c\nu^\f13\lambda(k-l)t}\iklt^4\hLw_{k-l}(\xi-\eta)\|_{L^2_{k-l,\xi-\eta}}\\
\lesssim & \langle t\rangle^{-2}\|e^{c\nu^\f13\lambda(D_x) t}\iDx\ifD^\epsilon\NLw\|_{L^2} \|e^{c\nu^\f13\lambda(D_x)t}\iDx^4\ifD^4\Lw\|_{L^2}^2.
\end{align*}

Combining the above two cases, we have proven
\begin{equation}\label{eq4.19}
\begin{aligned}
&\Big|\bigl( \Lu\cdot\nabla \Lw \big| \cM e^{2c\nu^\f13\lambda(D_x) t}\iDx^{2}\ifD^{2\epsilon} \NLw \bigr)\Big| \\
\lesssim & \langle t\rangle^{-2}\|e^{c\nu^\f13\lambda(D_x) t}\iDx\ifD^\epsilon\NLw\|_{L^2} \|e^{c\nu^\f13\lambda(D_x)t}\iDx^4\ifD^4\Lw\|_{L^2}^2.
\end{aligned}
\end{equation}

\begin{itemize}
\item For the treatment of the transport term $\Lu\cdot\nabla\NLw$.
\end{itemize}

Here, we need to use the point-wise inviscid damping of $\Lu$. In view of the different decay rates for the two components of $\Lu$, we first deal with the following term with $\Lu_2\p_y\NLw$:
\begin{align*}
&\Big|\bigl( \Lu_2\p_y \NLw \big| \cM e^{2c\nu^\f13\lambda(D_x) t}\iDx^{2}\ifD^{2\epsilon} \NLw \bigr)\Big|\\
=&\Big|\int_{\R^4} \cM(t,k,\xi)e^{2c\nu^\f13\lambda(k)t}\ik^2 \ifk^{2\epsilon} \hNLw_k(\xi) \f{l(\xi-\eta)}{|l|^2+|\eta|^2} \hLw_l(\eta) \hNLw_{k-l}(\xi-\eta) dkdld\xi d\eta \Big|.
\end{align*}
By applying \eqref{pointwise inviscid damping}, one has that for $|k|<1$,
\begin{align*}
&\Big|\int_{|k|<1} \cM(t,k,\xi)e^{2c\nu^\f13\lambda(k)t}\ik^2 \ifk^{2\epsilon} \hNLw_k(\xi) \f{l(\xi-\eta)}{|l|^2+|\eta|^2} \hLw_l(\eta) \hNLw_{k-l}(\xi-\eta) dkdld\xi d\eta \Big|\\
\lesssim & \|\ifk^\epsilon\|_{L^2_k([-1,1])}\|e^{c\nu^\f13\lambda(k)t}\ifk^\epsilon\hNLw_k(\xi)\|_{L^2_{k,\xi}} \langle t\rangle^{-2}\|\langle \eta+lt\rangle^{-1}\|_{L^2_\eta}\\
&\qquad\qquad \times \|e^{c\nu^\f13\lambda(l)t}\ilt^3\ifl^3\hLw_l(\eta)\|_{L^2_{l,\eta}}\|e^{c\nu^\f13\lambda(k-l)t} (\xi-\eta)\hNLw_{k-l}(\xi-\eta)\|_{L^2_{k-l,\xi-\eta}}\\
\lesssim &\langle t\rangle^{-2}\|e^{c\nu^\f13\lambda(D_x)t} \ifD^\epsilon\NLw\|_{L^2}
\|e^{c\nu^\f13\lambda(D_x)t} \iDt^3\ifD^3\Lw\|_{L^2}\|e^{c\nu^\f13\lambda(D_x)t} \p_y\NLw\|_{L^2},
\end{align*}
and for $|k|\geq 1$,
\begin{align*}
&\Big|\int_{|k|\geq 1} \cM(t,k,\xi)e^{2c\nu^\f13\lambda(k)t}\ik^2 \ifk^{2\epsilon} \hNLw_k(\xi) \f{l(\xi-\eta)}{|l|^2+|\eta|^2} \hLw_l(\eta) \hNLw_{k-l}(\xi-\eta) dkdld\xi d\eta \Big|\\
\lesssim & \|e^{c\nu^\f13\lambda(k)t}\ik\ifk^\epsilon\hNLw_k(\xi)\|_{L^2_{k,\xi}} \langle t\rangle^{-2}\|\ilt^{-2}\|_{L^2_{l,\eta}}\|e^{c\nu^\f13\lambda(l)t}\ilt^4\ifl^3\hLw_l(\eta)\|_{L^2_{l,\eta}}\\
&\qquad\qquad \times \|e^{c\nu^\f13\lambda(k-l)t} \ikl(\xi-\eta)\hNLw_{k-l}(\xi-\eta)\|_{L^2_{k-l,\xi-\eta}}\\
\lesssim &\langle t\rangle^{-2}\|e^{c\nu^\f13\lambda(D_x)t} \iDx\ifD^\epsilon\NLw\|_{L^2}
\|e^{c\nu^\f13\lambda(D_x)t} \iDt^4\ifD^3\Lw\|_{L^2}\\
&\qquad\qquad\times\|e^{c\nu^\f13\lambda(D_x)t}\iDx \p_y\NLw\|_{L^2}.
\end{align*}
The above two cases share the same estimate:
\begin{align*}
&\Big|\bigl( \Lu_2\p_y \NLw \big| \cM e^{2c\nu^\f13\lambda(D_x) t}\iDx^{2}\ifD^{2\epsilon} \NLw \bigr)\Big|\\
\lesssim &\nu^{-\f12}\langle t\rangle^{-2}\|e^{c\nu^\f13\lambda(D_x)t} \iDx\ifD^\epsilon\NLw\|_{L^2}
\|e^{c\nu^\f13\lambda(D_x)t} \iDt^4\ifD^3\Lw\|_{L^2}\\
&\qquad\qquad\times\nu^\f12\|e^{c\nu^\f13\lambda(D_x)t}\iDx \p_y\NLw\|_{L^2}.
\end{align*}

For the term with $\Lu_1\p_x\NLw$, we write
\begin{align*}
&\Big|\bigl( \Lu_1\p_x \NLw \big| \cM e^{2c\nu^\f13\lambda(D_x) t}\iDx^{2}\ifD^{2\epsilon} \NLw \bigr)\Big|\\
=&\Big|\int_{\R^4} \cM(t,k,\xi)e^{2c\nu^\f13\lambda(k)t}\ik^2 \ifk^{2\epsilon} \hNLw_k(\xi) \f{\eta(k-l)}{|l|^2+|\eta|^2} \hLw_l(\eta) \hNLw_{k-l}(\xi-\eta) dkdld\xi d\eta \Big|.
\end{align*}
From $\f\eta{|l|^2+|\eta|^2}$, we can only use \eqref{pointwise inviscid damping} to get $\langle t\rangle^{-1}\leq \nu^\f16$, but we shall see this is enough. When $|k-l|\leq 2|k|$, we can use $|k-l|^\f13\lesssim|k|^\f13$ and $\ifk^\epsilon\lesssim \ifkl^\epsilon$ to get
\begin{align*}
&\Big|\int_{|k-l|\leq 2|k|} \cM(t,k,\xi)e^{2c\nu^\f13\lambda(k)t}\ik^2 \ifk^{2\epsilon} \hNLw_k(\xi) \f{\eta(k-l)}{|l|^2+|\eta|^2} \hLw_l(\eta) \hNLw_{k-l}(\xi-\eta) dkdld\xi d\eta \Big|\\
\lesssim & \|e^{c\nu^\f13\lambda(k)t}\ik\ifk^\epsilon|k|^\f13\hNLw_k(\xi)\|_{L^2_{k,\xi}} \langle t\rangle^{-1}\|\ilt^{-2}\|_{L^2_{l,\eta}}\|e^{c\nu^\f13\lambda(l)t}\ilt^4\ifl^2\hLw_l(\eta)\|_{L^2_{l,\eta}}\\
&\qquad\qquad \times \|e^{c\nu^\f13\lambda(k-l)t} \ikl \ifkl^\epsilon|k-l|^\f23\hNLw_{k-l}(\xi-\eta)\|_{L^2_{k-l,\xi-\eta}}\\
\lesssim &\nu^{-\f13}\|e^{c\nu^\f13\lambda(D_x)t} \iDt^4\ifD^2\Lw\|_{L^2}\bigl(\nu^\f16\|e^{c\nu^\f13\lambda(D_x)t} \iDx\ifD^\epsilon|D_x|^\f13\NLw\|_{L^2}\bigr)^\f32\\
&\qquad\qquad\times\bigl(\nu^\f12\|e^{c\nu^\f13\lambda(D_x)t}\iDx\ifD^\epsilon  \p_x\NLw\|_{L^2}\bigr)^\f12.
\end{align*}
When $2|k|<|k-l|$, one has $\f{|k-l|}2\leq |l|\leq 2|k-l|$, and therefore $\ik\lesssim \il$. Then we have
\begin{align*}
&\Big|\int_{ 2|k|<|k-l|} \cM(t,k,\xi)e^{2c\nu^\f13\lambda(k)t}\ik^2 \ifk^{2\epsilon} \hNLw_k(\xi) \f{\eta(k-l)}{|l|^2+|\eta|^2} \hLw_l(\eta) \hNLw_{k-l}(\xi-\eta) dkdld\xi d\eta \Big|\\
\lesssim & \|\ik^{-1}\ifk^\epsilon\|_{L^2_k}\|e^{c\nu^\f13\lambda(k)t}\ik\ifk^\epsilon\hNLw_k(\xi)\|_{L^2_{k,\xi}} \langle t\rangle^{-1}\|e^{c\nu^\f13\lambda(l)t}\ilt^4\ifl^2\hLw_l(\eta)\|_{L^2_{l,\eta}}\\
&\qquad\qquad \times \|\langle \eta+lt\rangle^{-1}\|_{L^2_\eta}\|e^{c\nu^\f13\lambda(k-l)t} |k-l|^\f13\hNLw_{k-l}(\xi-\eta)\|_{L^2_{k-l,\xi-\eta}}\\
\lesssim &\nu^{-\f16}\langle t\rangle^{-1}\|e^{c\nu^\f13\lambda(D_x)t}\iDx\ifD^\epsilon  \NLw\|_{L^2}\|e^{c\nu^\f13\lambda(D_x)t} \iDt^4\ifD^2\Lw\|_{L^2}\\
&\qquad\qquad\times\nu^\f16\|e^{c\nu^\f13\lambda(D_x)t} |D_x|^\f13\NLw\|_{L^2}.
\end{align*}
Combing the above two cases, we arrive at
\begin{align*}
&\Big|\bigl( \Lu_1\p_x \NLw \big| \cM e^{2c\nu^\f13\lambda(D_x) t}\iDx^{2}\ifD^{2\epsilon} \NLw \bigr)\Big|\\
\lesssim &\langle t\rangle^{-2}\|e^{c\nu^\f13\lambda(D_x)t} \iDt^4\ifD^2\Lw\|_{L^2}\|e^{c\nu^\f13\lambda(D_x)t}\iDx\ifD^\epsilon  \NLw\|_{L^2}^2\\
&+\nu^{-\f13}\|e^{c\nu^\f13\lambda(D_x)t} \iDt^4\ifD^2\Lw\|_{L^2}\\
&\qquad\qquad\times\bigl(\nu^\f13\|e^{c\nu^\f13\lambda(D_x)t} \iDx\ifD^\epsilon|D_x|^\f13\NLw\|_{L^2}^2+\nu\|e^{c\nu^\f13\lambda(D_x)t}\iDx\ifD^\epsilon  \p_x\NLw\|_{L^2}^2\bigr).
\end{align*}

As a summary, for the transport term, we proved that
\begin{equation}\label{eq4.20}
\begin{aligned}
&\Big|\bigl( \Lu\cdot \nabla \NLw \big| \cM e^{2c\nu^\f13\lambda(D_x) t}\iDx^{2}\ifD^{2\epsilon} \NLw \bigr)\Big|\\
\lesssim &\nu^{-\f13}\langle t\rangle^{-2}\|e^{c\nu^\f13\lambda(D_x)t} \iDt^4\ifD^3\Lw\|_{L^2}\|e^{c\nu^\f13\lambda(D_x)t}\iDx\ifD^\epsilon  \NLw\|_{L^2}^2\\
&+\nu^{-\f13}\|e^{c\nu^\f13\lambda(D_x)t} \iDt^4\ifD^3\Lw\|_{L^2}\\
&\quad\times\bigl(\nu^\f13\|e^{c\nu^\f13\lambda(D_x)t} \iDx\ifD^\epsilon|D_x|^\f13\NLw\|_{L^2}^2+\nu\|e^{c\nu^\f13\lambda(D_x)t}\iDx\ifD^\epsilon  \nabla\NLw\|_{L^2}^2\bigr).
\end{aligned}
\end{equation}

\begin{itemize}
\item For the treatment of the reaction term $\NLu\cdot\nabla\Lw$.
\end{itemize}

For this term, we decompose it into two parts as follows:
\begin{align*}
&\Big|\bigl( \NLu\cdot\nabla \Lw \big| \cM e^{2c\nu^\f13\lambda(D_x) t}\iDx^{2}\ifD^{2\epsilon} \NLw \bigr)\Big|\\
=&\Big|\int_{\R^4} \cM(t,k,\xi)e^{2c\nu^\f13\lambda(k)t}\ik^2 \ifk^{2\epsilon} \hNLw_k(\xi) \f{\eta(k-l)-l(\xi-\eta)}{|l|^2+|\eta|^2} \hNLw_l(\eta) \hLw_{k-l}(\xi-\eta) dkdld\xi d\eta \Big|\\
\leq & I_1+I_2,
\end{align*}
where $I_1$ and $I_2$ are defined by 
\begin{align*}
I_1 =& \Big|\int_{\R^4} \cM(t,k,\xi)e^{2c\nu^\f13\lambda(k)t}\ik^2 \ifk^{2\epsilon} \hNLw_k(\xi) \f{\eta(k-l)-l(\xi-\eta+t(k-l))}{|l|^2+|\eta|^2} \hNLw_l(\eta) \\
&\qquad\qquad\qquad\qquad\times\hLw_{k-l}(\xi-\eta) dkdld\xi d\eta \Big|,\\
I_2 =&\Big|\int_{\R^4} \cM(t,k,\xi)e^{2c\nu^\f13\lambda(k)t}\ik^2 \ifk^{2\epsilon} \hNLw_k(\xi) \f{tl(k-l)}{|l|^2+|\eta|^2} \hNLw_l(\eta) \hLw_{k-l}(\xi-\eta) dkdld\xi d\eta \Big|.
\end{align*}

For $I_1$, $|k-l|+|\xi-\eta+t(k-l)|\leq \iklt$ can be handled by $\Lw$, the only possible difficulty is the singularity of $\f1{|l|+|\eta|}$. In the region $|k-l|\leq 2|l|$, we can derive that for $|k|>1$, 
\begin{align*}
&\Big|\int_{|k-l|\leq 2|l|,\ |k|>1} \cM(t,k,\xi)e^{2c\nu^\f13\lambda(k)t}\ik^2 \ifk^{2\epsilon} \hNLw_k(\xi) \f{\eta(k-l)-l(\xi-\eta+t(k-l))}{|l|^2+|\eta|^2}  \\
&\qquad\qquad\qquad\qquad\times\hNLw_l(\eta)\hLw_{k-l}(\xi-\eta) dkdld\xi d\eta \Big|\\
\lesssim & \|e^{c\nu^\f13\lambda(k)t} \ik \hNLw_k(\xi)\|_{L^2_{k,\xi}}
\|e^{c\nu^\f13\lambda(l)t} \il \f{l}{|l|+|\eta|} \hNLw_l(\eta)\|_{L^2_{l,\eta}} \|\iklt^{-2}\|_{L^2_{k-l,\xi-\eta}}\\
&\qquad\qquad\times \|e^{c\nu^\f13\lambda(k-l)t}\iklt^{4} \ifkl^2 |k-l|^\f13 \hLw_{k-l}(\xi-\eta)\|_{L^2_{k-l,\xi-\eta}}\\
\lesssim & \|e^{c\nu^\f13\lambda(D_x)t}\iDx\NLw\|_{L^2}\|e^{c\nu^\f13\lambda(D_x)t}\iDx\p_x\nabla\NLp\|_{L^2}\\
&\qquad\qquad\times\|e^{c\nu^\f13\lambda(D_x)t}\iDt^4\ifD^2|D_x|^\f13\Lw\|_{L^2},
\end{align*}
and for $|k|\leq 1$, 
\begin{align*}
&\Big|\int_{|k-l|\leq 2|l|,\ |k|\leq 1} \cM(t,k,\xi)e^{2c\nu^\f13\lambda(k)t}\ik^2 \ifk^{2\epsilon} \hNLw_k(\xi) \f{\eta(k-l)-l(\xi-\eta+t(k-l))}{|l|^2+|\eta|^2}  \\
&\qquad\qquad\qquad\qquad\times\hNLw_l(\eta)\hLw_{k-l}(\xi-\eta) dkdld\xi d\eta \Big|\\
\lesssim & \|\ifk^\epsilon\|_{L^2_k([-1,1])}\|e^{c\nu^\f13\lambda(k)t} \ik\ifk^\epsilon \hNLw_k(\xi)\|_{L^2_{k,\xi}}
\|e^{c\nu^\f13\lambda(l)t} \il \f{l}{|l|+|\eta|} \hNLw_l(\eta)\|_{L^2_{l,\eta}} \\
&\qquad\qquad\times \|\langle \xi-\eta+t(k-l)\rangle^{-1}\|_{L^2_{\xi-\eta}}\\
&\qquad\qquad\times\|e^{c\nu^\f13\lambda(k-l)t}\iklt^{3} \ifkl^2 |k-l|^\f13 \hLw_{k-l}(\xi-\eta)\|_{L^2_{k-l,\xi-\eta}}\\
\lesssim & \|e^{c\nu^\f13\lambda(D_x)t}\iDx\ifD^\epsilon\NLw\|_{L^2}\|e^{c\nu^\f13\lambda(D_x)t}\iDx\p_x\nabla\NLp\|_{L^2}\\
&\qquad\qquad\times\|e^{c\nu^\f13\lambda(D_x)t}\iDt^3\ifD^2|D_x|^\f13\Lw\|_{L^2}.
\end{align*}
When $2|l|<|k-l|$, one has $\f{|k-l|}2\leq|k|\leq 2|k-l|$ and therefore $\ifk\lesssim \ifkl$. Then, we have
\begin{align*}
&\Big|\int_{2|l|<|k-l|} \cM(t,k,\xi)e^{2c\nu^\f13\lambda(k)t}\ik^2 \ifk^{2\epsilon} \hNLw_k(\xi) \f{\eta(k-l)-l(\xi-\eta+t(k-l))}{|l|^2+|\eta|^2}  \\
&\qquad\qquad\qquad\qquad\times\hNLw_l(\eta)\hLw_{k-l}(\xi-\eta) dkdld\xi d\eta \Big|\\
\lesssim & \|e^{c\nu^\f13\lambda(k)t} \ik |k|^\f13 \hNLw_k(\xi)\|_{L^2_{k,\xi}}
\|e^{c\nu^\f13\lambda(l)t} \il \ifl^\epsilon \hNLw_l(\eta)\|_{L^2_{l,\eta}} \big\|\il^{-1}\ifl^{-\epsilon}\|\f1{|l|+|\eta|}\|_{L^2_\eta}\big\|_{L^2_l} \\
&\qquad\qquad\times\|e^{c\nu^\f13\lambda(k-l)t}\iklt^{2} \ifkl^2 |k-l|^\f13 \hLw_{k-l}(\xi-\eta)\|_{L^2_{k-l,\xi-\eta}}\\
\lesssim & \|e^{c\nu^\f13\lambda(D_x)t}\iDx|D_x|^\f13\NLw\|_{L^2}\|e^{c\nu^\f13\lambda(D_x)t}\iDx\ifD^\epsilon\NLw\|_{L^2}\\
&\qquad\qquad\times\|e^{c\nu^\f13\lambda(D_x)t}\iDt^2\ifD^2|D_x|^\f13\Lw\|_{L^2},
\end{align*}
where we used that for $\epsilon>0$,
$$
\big\|\il^{-1}\ifl^{-\epsilon}\|\f1{|l|+|\eta|}\|_{L^2_\eta}\big\|_{L^2_l}\lesssim \|\il^{-1}\ifl^{-\epsilon}|l|^{-\f12}\|_{L^2_l}\lesssim 1.
$$
Combining the above estimates, we conclude that
\begin{align*}
I_1\lesssim & \nu^{-\f16}\bigl(\|e^{c\nu^\f13\lambda(D_x)t}\iDx\p_x\nabla\NLp\|_{L^2}+\nu^\f16\|e^{c\nu^\f13\lambda(D_x)t}\iDx|D_x|^\f13\NLw\|_{L^2}\bigr)\\
&\qquad\times\|e^{c\nu^\f13\lambda(D_x)t}\iDx\ifD^\epsilon\NLw\|_{L^2}\|e^{c\nu^\f13\lambda(D_x)t}\iDt^4\ifD^2|D_x|^\f13\Lw\|_{L^2}.
\end{align*}

Now, we turn to the real reaction term $I_2$ which is the most difficult term in this paper. First, we consider the region $|k-l|\leq 2|k|$ and derive from $\ifk^\epsilon\lesssim \ifkl^\epsilon$ that
\begin{align*}
&\Big|\int_{|k-l|\leq 2|k|} \cM(t,k,\xi)e^{2c\nu^\f13\lambda(k)t}\ik^2 \ifk^{2\epsilon} \hNLw_k(\xi) \f{tl(k-l)}{|l|^2+|\eta|^2} \hNLw_l(\eta) \hLw_{k-l}(\xi-\eta) dkdld\xi d\eta \Big|\\
\lesssim &\|e^{c\nu^\f13\lambda(l)t}\il\ifl^\epsilon\f{l}{\sqrt{|l|^2+|\eta|^2}}\hNLw_{l}(\eta)\|_{L^2_{l,\eta}}\\
&\qquad\qquad\times\|te^{c\nu^\f13\lambda(k-l)t}\iklt^3 \ifkl\hLw_{k-l}(\xi-\eta)\|_{L^2_{k-l,\xi-\eta}}\\
&\quad\times \bigl(\int_{\R^4} e^{2c\nu^\f13\lambda(k)t} \ik^2 \ifk^{2\epsilon}|k|^{\f23\delta} \f{\ifl^{-2\epsilon}}{(|l|^2+|\eta|^2)\iklt^2}|\hNLw_k(\xi)|^2 dkdld\xi d\eta\bigr)^\f12.
\end{align*}
Then, on the one hand, one can obtain from $\lambda(k-l)\leq |k-l|^\f23$ that
\begin{equation}\label{eq4.21}
\begin{aligned}
&\|te^{c\nu^\f13\lambda(k-l)t}\iklt^3 \ifkl\hLw_{k-l}(\xi-\eta)\|_{L^2_{k-l,\xi-\eta}} \\
\lesssim & \nu^{-\f13} \|c\nu^\f13|k-l|^\f23 t e^{c\nu^\f13|k-l|^\f23t}\iklt^3 \ifkl^2 \hLw_{k-l}(\xi-\eta)\|_{L^2_{k-l,\xi-\eta}} \\
\lesssim & \nu^{-\f13} \|e^{c_0\nu |k-l|^2t^3}\iklt^3 \ifkl^2 \hLw_{k-l}(\xi-\eta)\|_{L^2_{k-l,\xi-\eta}}.
\end{aligned}
\end{equation}
On the other hand, one can use the convolution properties of the Poisson kernel to compute
\begin{equation}\label{eq4.22}
\begin{aligned}
&\int_{\R^2} \f{\ifl^{-2\epsilon}}{(|l|^2+|\eta|^2)\iklt^2} d\eta dl \\
\lesssim & \int_{\R} \ifl^{-2\epsilon}\f1{|l|} \f{1+|l|+|k-l|}{(1+|k-l|)(\langle l,k-l\rangle^2+|\xi+t(k-l)|^2)}dl \\
\lesssim & \bigl( \int_{\R} \ifl^{-1-\kappa}\f1{|l|} \f{1+|l|+|k-l|}{(1+|k-l|)(\langle l,k-l\rangle^2+|\xi+t(k-l)|^2)}dl \bigr)^{1-\delta}
\bigl( \int_{\R} \ifl^{-\alpha}\f1{|l|} \langle l\rangle^{-1} dl \bigr)^{\delta}\\
\lesssim & \Upsilon(t,k,\xi)^{1-\delta},
\end{aligned}
\end{equation}
where $\alpha=\f{2\epsilon-(1-\delta)(1+\kappa)}\delta$ satisfies
$$
(1-\delta)(-1-\kappa)+\delta(-\alpha)=-2\epsilon,
$$
and we can choose $0<\kappa<\f{2\epsilon}{1-\delta}-1$ so that $\alpha>0$, when $\epsilon>\f{1-\delta}2$. Using the above two estimates, we arrive at
\begin{align*}
&\Big|\int_{|k-l|\leq 2|k|} \cM(t,k,\xi)e^{2c\nu^\f13\lambda(k)t}\ik^2 \ifk^{2\epsilon} \hNLw_k(\xi) \f{tl(k-l)}{|l|^2+|\eta|^2} \hNLw_l(\eta) \hLw_{k-l}(\xi-\eta) dkdld\xi d\eta \Big|\\
\lesssim &\|e^{c\nu^\f13\lambda(l)t}\il\ifl^\epsilon\f{l}{\sqrt{|l|^2+|\eta|^2}}\hNLw_{l}(\eta)\|_{L^2_{l,\eta}}\\
&\qquad\qquad\times\nu^{-\f13} \|e^{c_0\nu|k-l|^2t^3}\iklt^3 \ifkl^2 \hLw_{k-l}(\xi-\eta)\|_{L^2_{k-l,\xi-\eta}}\\
&\qquad\qquad\qquad\qquad\times 
\| e^{c\nu^\f13\lambda(k)t} \ik \ifk^\epsilon |k|^{\f\delta3} \Upsilon(t,k,\xi)^{\f{1-\delta}2} \hNLw_k(\xi)\|_{L^2_{k,\xi}}\\
\lesssim &\nu^{-\f13-\f{\delta}6} \|e^{c\nu^\f13\lambda(D_x)}\iDx\ifD^\epsilon \p_x\nabla \NLp\|_{L^2} \|e^{c_0\nu |D_x|^2t^3}\iDt^3\ifD^2\Lw\|_{L^2}\\
&\qquad\qquad\times \bigl( \nu^\f16 \|e^{c\nu^\f13\lambda(D_x)t}\iDx\ifD^\epsilon|D_x|^\f13 \NLw\|_{L^2} \bigr)^\delta \|e^{c\nu^\f13\lambda(D_x)t}\iDx\ifD^\epsilon \sqrt{\Upsilon} \NLw\|_{L^2}^{1-\delta}.
\end{align*}
When $2|k|<|k-l|$, one has $\f{|k-l|}2\leq|l|\leq 2|k-l|$ and therefore
$\ifl\lesssim \ifkl$. In such a case, we do not need to use the tricks in \eqref{eq4.22}, and it follows from a similar argument as \eqref{eq4.21} that
\begin{align*}
&\Big|\int_{2|k|<|k-l|} \cM(t,k,\xi)e^{2c\nu^\f13\lambda(k)t}\ik^2 \ifk^{2\epsilon} \hNLw_k(\xi) \f{tl(k-l)}{|l|^2+|\eta|^2} \hNLw_l(\eta) \hLw_{k-l}(\xi-\eta) dkdld\xi d\eta \Big|\\
\lesssim &\|\ik^{-1}\ifk^\epsilon\|_{L^2_k}\|e^{c\nu^\f13\lambda(l)t}\il\ifl^\epsilon\f{l}{\sqrt{|l|^2+|\eta|^2}}\hNLw_{l}(\eta)\|_{L^2_{l,\eta}}\|\langle \xi-\eta+t(k-l) \rangle^{-1}\|_{L^2_\xi}\\
&\qquad\qquad\times\|te^{c\nu^\f13\lambda(k-l)t}\iklt^4 \ifkl\hLw_{k-l}(\xi-\eta)\|_{L^\oo_{k-l,\xi-\eta}}\\
&\quad\times \bigl(\int_{\R^4} e^{2c\nu^\f13\lambda(k)t} \ik^2 \ifk^{2\epsilon} \f{\ifl^{-1-\kappa}}{(|l|^2+|\eta|^2)\iklt^2}|\hNLw_k(\xi)|^2 dkdld\xi d\eta\bigr)^\f12\\
\lesssim &\|e^{c\nu^\f13\lambda(l)t}\il\ifl^\epsilon\f{l}{\sqrt{|l|^2+|\eta|^2}}\hNLw_{l}(\eta)\|_{L^2_{l,\eta}}\\
&\qquad\qquad\times\nu^{-\f13} \|e^{c_0\nu|k-l|^2t^3}\iklt^4 \ifkl^2 \hLw_{k-l}(\xi-\eta)\|_{L^\oo_{k-l,\xi-\eta}}\\
&\qquad\qquad\qquad\qquad\times 
\| e^{c\nu^\f13\lambda(k)t} \ik \ifk^\epsilon  \sqrt{\Upsilon(t,k,\xi)} \hNLw_k(\xi)\|_{L^2_{k,\xi}}\\
\lesssim &\nu^{-\f13} \|e^{c\nu^\f13\lambda(D_x)}\iDx\ifD^\epsilon \p_x\nabla \NLp\|_{L^2} \|e^{c_0\nu|D_x|^2t^3}\iDt^4\ifD^2\Lw\|_{L^1}\\
&\qquad\qquad\times  \|e^{c\nu^\f13\lambda(D_x)t}\iDx\ifD^\epsilon \sqrt{\Upsilon} \NLw\|_{L^2}.
\end{align*}
Combining the above two cases, we arrive at
\begin{align*}
I_2 \lesssim & \nu^{-\f13-\f\delta6}\|e^{c\nu^\f13\lambda(D_x)}\iDx\ifD^\epsilon \p_x\nabla \NLp\|_{L^2} \|e^{c_0\nu|D_x|^2t^3}\iDt^4\ifD^2\Lw\|_{L^1\cap L^2}\\
&\qquad\qquad\times \bigl(\|e^{c\nu^\f13\lambda(D_x)t}\iDx\ifD^\epsilon \sqrt{\Upsilon} \NLw\|_{L^2}+ \nu^\f16 \|e^{c\nu^\f13\lambda(D_x)t}\iDx\ifD^\epsilon|D_x|^\f13 \NLw\|_{L^2} \bigr).
\end{align*}

In all, we add together the estimates for $I_1$ and $I_2$ to summarise 
\begin{equation}\label{eq4.23}
\begin{aligned}
&\Big|\bigl( \NLu\cdot\nabla \Lw \big| \cM e^{2c\nu^\f13\lambda(D_x) t}\iDx^{2}\ifD^{2\epsilon} \NLw \bigr)\Big|\\
\lesssim &\nu^{-\f16}\bigl(\|e^{c\nu^\f13\lambda(D_x)t}\iDx\p_x\nabla\NLp\|_{L^2}+\nu^\f16\|e^{c\nu^\f13\lambda(D_x)t}\iDx|D_x|^\f13\NLw\|_{L^2}\bigr)\\
&\qquad\times\|e^{c\nu^\f13\lambda(D_x)t}\iDx\ifD^\epsilon\NLw\|_{L^2}\|e^{c\nu^\f13\lambda(D_x)t}\iDt^4\ifD^2|D_x|^\f13\Lw\|_{L^2}\\
& +\nu^{-\f13-\f\delta6}\|e^{c\nu^\f13\lambda(D_x)}\iDx\ifD^\epsilon \p_x\nabla \NLp\|_{L^2} \|e^{c_0\nu|D_x|^2t^3}\iDt^4\ifD^2\Lw\|_{L^1\cap L^2}\\
&\qquad\times \bigl(\|e^{c\nu^\f13\lambda(D_x)t}\iDx\ifD^\epsilon \sqrt{\Upsilon} \NLw\|_{L^2}+ \nu^\f16 \|e^{c\nu^\f13\lambda(D_x)t}\iDx\ifD^\epsilon|D_x|^\f13 \NLw\|_{L^2} \bigr).
\end{aligned}
\end{equation}

\begin{itemize}
\item For the treatment of the nonlinear term $\NLu\cdot\nabla\NLw$.
\end{itemize}

The only differences between this term and the nonlinear term in \eqref{eq3.1} with $m=1$ are the enhanced dissipation rate between $\lambda(D_x)$ and $|D_x|^\f23$, as well as the Fourier multiplier $\cM$. However, in the proof of \eqref{eq3.3} and \eqref{eq3.4}, we only used $|k|^\f23\leq |l|^\f23+|k-l|^\f23$ and the boundedness of $\cM$. Since $\lambda(k)\leq \lambda(l)+\lambda(k-l)$ and the $\cM$ here is still bounded, one can follow the arguments in \eqref{eq3.3} and \eqref{eq3.4}, to prove
\begin{equation}\label{eq4.24}
\begin{aligned}
&\Big|\bigl( \NLu\cdot\nabla \NLw \big| \cM e^{2c\nu^\f13\lambda(D_x) t}\iDx^{2}\ifD^{2\epsilon} \NLw \bigr)\Big| \\
\lesssim &  \nu^{-\f12}\|e^{c\nu^\f13\lambda(D_x) t}\iDx\ifD^{\epsilon} \NLw\|_{L^2}
\bigl( \|e^{c\nu^\f13\lambda(D_x) t}\iDx\ifD^{\epsilon} \p_x\nabla\NLp\|_{L^2}^2  \\
&\qquad+\nu\|e^{c\nu^\f13\lambda(D_x) t}\iDx\ifD^{\epsilon} \nabla\NLw\|_{L^2}^2 + \nu^\f13\|e^{c\nu^\f13\lambda(D_x) t}\iDx\ifD^{\epsilon} |D_x|^\f13\NLw\|_{L^2}^2\bigr).
\end{aligned}
\end{equation}

Now, we take \eqref{eq4.19}, \eqref{eq4.20}, \eqref{eq4.23} and \eqref{eq4.24} into \eqref{eq4.18}, and then integrate over $[T_0,T]$ to find
\begin{align*}
&\|e^{c\nu^\f13\lambda(D_x) t}\iDx\ifD^{\epsilon} \NLw\|_{L^\oo_{[T_0,T]}L^2}^2
+\nu\|e^{c\nu^\f13\lambda(D_x) t}\iDx\ifD^{\epsilon} \nabla \NLw\|_{L^2_{[T_0,T]}L^2}^2
 \\
&\quad+\nu^\f13 \|e^{c\nu^\f13\lambda(D_x) t}\iDx\ifD^{\epsilon} |D_x|^\f13\NLw\|_{L^2_{[T_0,T]}L^2}^2
+\|e^{c\nu^\f13\lambda(D_x) t}\iDx\ifD^{\epsilon} \p_x \nabla \NLp\|_{L^2_{[T_0,T]}L^2}^2 \\
&\qquad+\|e^{c\nu^\f13\lambda(D_x) t}\iDx\ifD^{\epsilon} \sqrt{\Upsilon(t,D_x,D_y)} \NLw\|_{L^2_{[T_0,T]}L^2}^2\\
&\lesssim \|e^{c\nu^\f13\lambda(D_x) T_0}\iDx\ifD^{\epsilon} \NLw(T_0)\|_{L^2}^2\\
&\qquad+\|\langle t\rangle^{-2}\|_{L^1_{[T_0,T]}}\|e^{c\nu^\f13\lambda(D_x) t}\iDx\ifD^\epsilon\NLw\|_{L^\oo_{[T_0,T]}L^2} \|e^{c\nu^\f13\lambda(D_x)t}\iDx^4\ifD^4\Lw\|_{L^\oo_{[T_0,T]}L^2}^2\\
&+\nu^{-\f13}\|\langle t\rangle^{-2}\|_{L^1_{[T_0,T]}}\|e^{c\nu^\f13\lambda(D_x)t} \iDt^4\ifD^3\Lw\|_{L^\oo_{[T_0,T]}L^2}\\
&\times\|e^{c\nu^\f13\lambda(D_x)t}\iDx\ifD^\epsilon  \NLw\|_{L^\oo_{[T_0,T]}L^2}^2+\nu^{-\f13}\|e^{c\nu^\f13\lambda(D_x)t} \iDt^4\ifD^3\Lw\|_{L^\oo_{[T_0,T]}L^2}\\
&\quad\times\bigl(\nu^\f13\|e^{c\nu^\f13\lambda(D_x)t} \iDx\ifD^\epsilon|D_x|^\f13\NLw\|_{L^2_{[T_0,T]}L^2}^2+\nu\|e^{c\nu^\f13\lambda(D_x)t}\iDx\ifD^\epsilon  \nabla\NLw\|_{L^2_{[T_0,T]}L^2}^2\bigr)\\
&+\nu^{-\f16}\bigl(\|e^{c\nu^\f13\lambda(D_x)t}\iDx\p_x\nabla\NLp\|_{L^2_{[T_0,T]}L^2}+\nu^\f16\|e^{c\nu^\f13\lambda(D_x)t}\iDx|D_x|^\f13\NLw\|_{L^2_{[T_0,T]}L^2}\bigr)\\
&\times\|e^{c\nu^\f13\lambda(D_x)t}\iDx\ifD^\epsilon\NLw\|_{L^\oo_{[T_0,T]}L^2}\|e^{c\nu^\f13\lambda(D_x)t}\iDt^4\ifD^2|D_x|^\f13\Lw\|_{L^2_{[T_0,T]}L^2}\\
& +\nu^{-\f13-\f\delta6}\|e^{c\nu^\f13\lambda(D_x)}\iDx\ifD^\epsilon \p_x\nabla \NLp\|_{L^2_{[T_0,T]}L^2} \\
&\qquad\times \|e^{c_0\nu|D_x|^2t^3}\iDt^4\ifD^2\Lw\|_{L^\oo_{[T_0,T]}(L^1\cap L^2)}\\
&\times \bigl(\|e^{c\nu^\f13\lambda(D_x)t}\iDx\ifD^\epsilon \sqrt{\Upsilon} \NLw\|_{L^2_{[T_0,T]}L^2}+ \nu^\f16 \|e^{c\nu^\f13\lambda(D_x)t}\iDx\ifD^\epsilon|D_x|^\f13 \NLw\|_{L^2_{[T_0,T]}L^2} \bigr)\\
&+\nu^{-\f12}\|e^{c\nu^\f13\lambda(D_x) t}\iDx\ifD^{\epsilon} \NLw\|_{L^\oo_{[T_0,T]}L^2}
\bigl( \|e^{c\nu^\f13\lambda(D_x) t}\iDx\ifD^{\epsilon} \p_x\nabla\NLp\|_{L^2_{[T_0,T]}L^2}^2  \\
&\quad+\nu\|e^{c\nu^\f13\lambda(D_x) t}\iDx\ifD^{\epsilon} \nabla\NLw\|_{L^2_{[T_0,T]}L^2}^2 + \nu^\f13\|e^{c\nu^\f13\lambda(D_x) t}\iDx\ifD^{\epsilon} |D_x|^\f13\NLw\|_{L^2_{[T_0,T]}L^2}^2\bigr).
\end{align*}
This together with \eqref{eq:lem4.1} and 
$
\|\langle t\rangle^{-2}\|_{L^1_{[T_0,T]}}\lesssim \nu^\f16
$
finishes the proof.
\end{proof}

\subsection{Proof of Theorem \ref{thm2}}

We are now in a position to present the proof of Theorem \ref{thm2}.

\begin{proof}[Proof of Theorem \ref{thm2}] In this proof, we fix $\delta$, $\epsilon$ and $c$ such that $\epsilon>\f{1-\delta}2$ and $c<c_0$, therefore the conditions of Proposition \ref{prop4.4} are fulfilled. We shall always denote $T_0=\nu^{-\f16}$.

From Corollary \ref{col4.1}, we know 
\begin{equation}\label{eq4.26}
\|e^{c\nu^\f13\lambda(D_x)t}\iDx \ifD^\epsilon \NLw\|_{L^\oo_{T_0} L^2}\leq C\nu^{\f23+2\delta}.
\end{equation}
Taking the above estimate into \eqref{eq:prop4.4}, we find
\begin{equation}\label{eq4.27}
\begin{aligned}
&\Bigl(1-C\nu^{\f16+\delta}\Bigr)\|e^{c\nu^\f13\lambda(D_x) t}\iDx\ifD^{\epsilon} \NLw\|_{L^\oo_{[T_0,T]}L^2}^2 \\
&+\Bigl(1-C\nu^{\f56\delta}-C\nu^{-\f12}\|e^{c\nu^\f13\lambda(D_x) t}\iDx\ifD^{\epsilon} \NLw\|_{L^\oo_{[T_0,T]}L^2}\Bigr)\\
&\quad\times\Bigl(\nu\|e^{c\nu^\f13\lambda(D_x) t}\iDx\ifD^{\epsilon} \nabla \NLw\|_{L^2_{[T_0,T]}L^2}^2+\nu^\f13 \|e^{c\nu^\f13\lambda(D_x) t}\iDx\ifD^{\epsilon} |D_x|^\f13\NLw\|_{L^2_{[T_0,T]}L^2}^2
\\
&\qquad+\|e^{c\nu^\f13\lambda(D_x) t}\iDx\ifD^{\epsilon} \p_x \nabla \NLp\|_{L^2_{[T_0,T]}L^2}^2+\|e^{c\nu^\f13\lambda(D_x) t}\iDx\ifD^{\epsilon} \sqrt{\Upsilon} \NLw\|_{L^2_{[T_0,T]}L^2}^2\Bigr)\\
&\leq C\nu^{\f43+4\delta}+C\nu^{\f56+2\delta}\|e^{c\nu^\f13\lambda(D_x) t}\iDx\ifD^\epsilon\NLw\|_{L^\oo_{[T_0,T]}L^2}.
\end{aligned}
\end{equation}

Now, we fix $C>1$ to be the large constant in the above inequality, and define the following timespan:
\begin{equation}\label{def:T3}
T^\# \eqdef \sup\big\{ T>T_0 \big|\ \|e^{c\nu^\f13\lambda(D_x) t}\iDx\ifD^{\epsilon} \NLw\|_{L^\oo_{[T_0,T]}L^2}\leq 4C\nu^{\f23+2\delta}\big\}.
\end{equation}
From \eqref{eq4.26}, we know $T^\#>T_0$.

Now, by taking $\nu$ small enough such that $C\nu^{\f16+\delta}\leq \f12$, and for $T\leq T^\#$
$$
C\nu^{\f56\delta}+C\nu^{-\f12}\|e^{c\nu^\f13\lambda(D_x) t}\iDx\ifD^{\epsilon} \NLw\|_{L^\oo_{[T_0,T]}L^2}
\leq C\nu^{\f56\delta}+4C^2\nu^{\f16+2\delta}\leq \f12,
$$
and
$$
C\nu^{\f56+2\delta}\|e^{c\nu^\f13\lambda(D_x) t}\iDx\ifD^\epsilon\NLw\|_{L^\oo_{[T_0,T]}L^2}
\leq 4C^2\nu^{\f32+4\delta}\leq C\nu^{\f43+4\delta}.
$$
Then, one can drive from \eqref{eq4.27} that for any $T\leq T^\#$,
\begin{equation}\label{eq4.29}
\begin{aligned}
&\|e^{c\nu^\f13\lambda(D_x) t}\iDx\ifD^{\epsilon} \NLw\|_{L^\oo_{[T_0,T]}L^2}^2 
+\nu\|e^{c\nu^\f13\lambda(D_x) t}\iDx\ifD^{\epsilon} \nabla \NLw\|_{L^2_{[T_0,T]}L^2}^2\\
&+\nu^\f13 \|e^{c\nu^\f13\lambda(D_x) t}\iDx\ifD^{\epsilon} |D_x|^\f13\NLw\|_{L^2_{[T_0,T]}L^2}^2
+\|e^{c\nu^\f13\lambda(D_x) t}\iDx\ifD^{\epsilon} \p_x \nabla \NLp\|_{L^2_{[T_0,T]}L^2}^2\\
&+\|e^{c\nu^\f13\lambda(D_x) t}\iDx\ifD^{\epsilon} \sqrt{\Upsilon} \NLw\|_{L^2_{[T_0,T]}L^2}^2
\leq 4C\nu^{\f43+4\delta},
\end{aligned}
\end{equation}
which implies
$$
\|e^{c\nu^\f13\lambda(D_x) t}\iDx\ifD^{\epsilon} \NLw\|_{L^\oo_{[T_0,T]}L^2}
\leq 2\sqrt{C}\nu^{\f23+2\delta}\leq 2C\nu^{\f23+2\delta}.
$$

Together with \eqref{def:T3}, one can use a standard bootstrap argument to show $T^\#=+\oo$. \eqref{eq1.7} follows from \eqref{eq:lem4.1}, \eqref{eq4.26} and \eqref{eq4.29}. Also, we can use \eqref{pointwise inviscid damping} and \eqref{eq:lem4.1} to see
$$
\| e^{c\nu^\f13 \lambda(D_x) t} \iDx \ifD^\epsilon \p_x \Lu(t)\|_{L^2}
\lesssim \langle t\rangle^{-1} \| e^{c\nu^\f13 \lambda(D_x) t} \iDx^4\ifD^3\Lw(t)\|_{L^2}\lesssim \nu^{\f13+\delta}\langle t\rangle^{-1},
$$
and it follows from \eqref{eq:col4.1} that
$$
\| e^{c\nu^\f13 \lambda(D_x) t} \iDx \ifD^\epsilon \p_x \NLu(t)\|_{L^2_{T_0}L^2}
\lesssim \sqrt{T_0}\| e^{c\nu^\f13 \lambda(D_x) t} \iDx \ifD^\epsilon  \NLw(t)\|_{L^\oo_{T_0}L^2}
\lesssim \nu^{\f7{12}+2\delta}.
$$
The above two estimates together with \eqref{eq4.29} prove \eqref{eq1.8}.
\end{proof}

\section*{Acknowledgments}
This work was done when N. Liu was visiting NYUAD in 2024, and he gratefully acknowledged the support received from AMSS and NYUAD during the visit.

\bibliographystyle{siam.bst} 
\bibliography{references.bib}

\begin{thebibliography}{10}

\bibitem{albritton2022enhanced}
{\sc D.~Albritton, R.~Beekie, and M.~Novack}, {\em Enhanced dissipation and
  h{\"o}rmander's hypoellipticity}, Journal of Functional Analysis, 283 (2022),
  p.~109522.

\bibitem{arbon2024}
{\sc R.~Arbon and J.~Bedrossian}, {\em Quantitative hydrodynamic stability for
  {C}ouette flow on unbounded domains with {N}avier boundary conditions}, arXiv
  preprint arXiv:2404.02412,  (2024).

\bibitem{aris1956dispersion}
{\sc R.~Aris}, {\em On the dispersion of a solute in a fluid flowing through a
  tube}, Proceedings of the Royal Society of London. Series A. Mathematical and
  Physical Sciences, 235 (1956), pp.~67--77.

\bibitem{BGM2017}
{\sc J.~Bedrossian, P.~Germain, and N.~Masmoudi}, {\em On the stability
  threshold for the 3{D} {C}ouette flow in {S}obolev regularity}, Ann. of Math.
  (2), 185 (2017), pp.~541--608.

\bibitem{BGM2020}
\leavevmode\vrule height 2pt depth -1.6pt width 23pt, {\em Dynamics near the
  subcritical transition of the 3{D} {C}ouette flow {I}: {B}elow threshold
  case}, Mem. Amer. Math. Soc., 266 (2020), pp.~v+158.

\bibitem{BGM2015}
\leavevmode\vrule height 2pt depth -1.6pt width 23pt, {\em Dynamics near the
  subcritical transition of the 3{D} {C}ouette flow {II}: {A}bove threshold
  case}, Mem. Amer. Math. Soc., 279 (2022), pp.~v+135.

\bibitem{BHIW2023}
{\sc J.~Bedrossian, S.~He, S.~Iyer, and F.~Wang}, {\em Stability threshold of
  nearly-couette shear flows with navier boundary conditions in 2d}, arXiv
  preprint arXiv:2311.00141,  (2023).

\bibitem{bedrossian2024uniform}
\leavevmode\vrule height 2pt depth -1.6pt width 23pt, {\em Uniform inviscid
  damping and inviscid limit of the 2d navier-stokes equation with navier
  boundary conditions}, arXiv preprint arXiv:2405.19249,  (2024).

\bibitem{BM2015}
{\sc J.~Bedrossian and N.~Masmoudi}, {\em Inviscid damping and the asymptotic
  stability of planar shear flows in the 2{D} {E}uler equations}, Publ. Math.
  Inst. Hautes \'{E}tudes Sci., 122 (2015), pp.~195--300.

\bibitem{BMV2016}
{\sc J.~Bedrossian, N.~Masmoudi, and V.~Vicol}, {\em Enhanced dissipation and
  inviscid damping in the inviscid limit of the {N}avier-{S}tokes equations
  near the two dimensional {C}ouette flow}, Arch. Ration. Mech. Anal., 219
  (2016), pp.~1087--1159.

\bibitem{BVW2018}
{\sc J.~Bedrossian, V.~Vicol, and F.~Wang}, {\em The {S}obolev stability
  threshold for 2{D} shear flows near {C}ouette}, J. Nonlinear Sci., 28 (2018),
  pp.~2051--2075.

\bibitem{CLWZ2020}
{\sc Q.~Chen, T.~Li, D.~Wei, and Z.~Zhang}, {\em Transition threshold for the
  2-{D} {C}ouette flow in a finite channel}, Arch. Ration. Mech. Anal., 238
  (2020), pp.~125--183.

\bibitem{ChenWeiZhang2020}
{\sc Q.~Chen, D.~Wei, and Z.~Zhang}, {\em Transition threshold for the 3{D}
  {C}ouette flow in a finite channel}, Mem. Amer. Math. Soc., 296 (2024).

\bibitem{ZelatiGallay2023}
{\sc M.~Coti~Zelati and T.~Gallay}, {\em Enhanced dissipation and taylor
  dispersion in higher-dimensional parallel shear flows}, Journal of the London
  Mathematical Society, 108 (2023), pp.~1358--1392.

\bibitem{DengWuZhang2021}
{\sc W.~Deng, J.~Wu, and P.~Zhang}, {\em Stability of {C}ouette flow for 2{D}
  {B}oussinesq system with vertical dissipation}, J. Funct. Anal., 281 (2021),
  pp.~Paper No. 109255, 40.

\bibitem{DM2023}
{\sc Y.~Deng and N.~Masmoudi}, {\em Long-time instability of the {C}ouette flow
  in low {G}evrey spaces}, Communications on Pure and Applied Mathematics, 76
  (2023), pp.~2804--2887.

\bibitem{IJ2018}
{\sc A.~D. Ionescu and H.~Jia}, {\em Inviscid damping near the {C}ouette flow
  in a channel}, Communications in Mathematical Physics, 374 (2020),
  pp.~2015--2096.

\bibitem{IJ2020}
{\sc A.~D. Ionescu and H.~Jia}, {\em Non-linear inviscid damping near monotonic
  shear flows}, Acta Math., 230 (2023), pp.~321--399.

\bibitem{Kelvin1887}
{\sc L.~Kelvin}, {\em Stability of fluid motion: rectilinear motion of viscous
  fluid between two parallel plates}, Phil. Mag, 24 (1887), pp.~188--196.

\bibitem{LMZ2022G}
{\sc H.~Li, N.~Masmoudi, and W.~Zhao}, {\em Asymptotic stability of
  two-dimensional couette flow in a viscous fluid}, arXiv preprint
  arXiv:2208.14898,  (2022).

\bibitem{LiMasmoudiZhao2022critical}
\leavevmode\vrule height 2pt depth -1.6pt width 23pt, {\em A dynamical approach
  to the study of instability near {C}ouette flow}, Comm. Pure Appl. Math., 77
  (2024), pp.~2863--2946.

\bibitem{MasmoudiZhao2020cpde}
{\sc N.~Masmoudi and W.~Zhao}, {\em Enhanced dissipation for the 2{D} {C}ouette
  flow in critical space}, Comm. Partial Differential Equations, 45 (2020),
  pp.~1682--1701.

\bibitem{MasmoudiZhao2019}
\leavevmode\vrule height 2pt depth -1.6pt width 23pt, {\em Stability threshold
  of two-dimensional {C}ouette flow in {S}obolev spaces}, Annales de l'Institut
  Henri Poincar\'e C, Analyse Non lin\'eaire, 39 (2022), pp.~245--325.

\bibitem{MasmoudiZhao2020}
\leavevmode\vrule height 2pt depth -1.6pt width 23pt, {\em Nonlinear inviscid
  damping for a class of monotone shear flows in a finite channel}, Ann. of
  Math. (2), 199 (2024), pp.~1093--1175.

\bibitem{niuzhao2024}
{\sc B.~Niu and W.~Zhao}, {\em Improved stability threshold of the
  two-dimensional {C}ouette flow for {N}avier-{S}tokes-{B}oussinesq systems via
  quasi-linearization}, arXiv preprint arXiv:2409.16216,  (2024).

\bibitem{Orr1907}
{\sc W.~M. Orr}, {\em The stability or instability of the steady motions of a
  perfect liquid and of a viscous liquid}, Proc. Ir. Acad. Sect. A, Math
  Astron. Phys. Sci, 27 (1907), pp.~9--68.

\bibitem{Rayleigh1880}
{\sc L.~Rayleigh}, {\em On the {S}tability, or {I}nstability, of certain
  {F}luid {M}otions}, Proc. Lond. Math. Soc., 11 (1879/80), pp.~57--70.

\bibitem{Sommerfeld1908}
{\sc A.~Sommerfeld}, {\em Ein beitrag zur hydrodynamischen erkl{\"a}rung der
  turbulenten fl{\"u}ssigkeitsbewegung}, Atti del IV Congresso internazionale
  dei matematici,  (1908), pp.~116--124.

\bibitem{taylor1954dispersion}
{\sc G.~I. Taylor}, {\em The dispersion of matter in turbulent flow through a
  pipe}, Proceedings of the Royal Society of London. Series A. Mathematical and
  Physical Sciences, 223 (1954), pp.~446--468.

\bibitem{wang2024transition}
{\sc G.~Wang and W.~Wang}, {\em Transition threshold for the 2-{D} {C}ouette
  flow in whole space via {G}reen's function}, arXiv preprint arXiv:2404.11878,
   (2024).

\bibitem{WeiZhang2020}
{\sc D.~Wei and Z.~Zhang}, {\em Transition threshold for the 3{D} {C}ouette
  flow in {S}obolev space}, Communications on Pure and Applied Mathematics,
  (2020).

\bibitem{wei2023nonlinear}
\leavevmode\vrule height 2pt depth -1.6pt width 23pt, {\em Nonlinear enhanced
  dissipation and inviscid damping for the 2{D} {C}ouette flow}, Tunisian
  Journal of Mathematics, 5 (2023), pp.~573--592.

\bibitem{Wei2023quasilinear}
\leavevmode\vrule height 2pt depth -1.6pt width 23pt, {\em Stability threshold
  of the 2d {C}ouette flow in a finite channel}, preprint,  (2023).

\bibitem{ZhaiZhao2022}
{\sc C.~Zhai and W.~Zhao}, {\em Stability threshold of the {C}ouette flow for
  {N}avier-{S}tokes {B}oussinesq system with large {R}ichardson number
  $\gamma^2>1/4$}, SIAM Journal on Mathematical Analysis, 55 (2023),
  pp.~1284--1318.

\end{thebibliography}

\end{document}